\documentclass[12pt, reqno]{amsart}
\usepackage{amsmath, amsthm, amscd, amsfonts, amssymb, graphicx, color}
\usepackage[bookmarksnumbered, colorlinks, plainpages]{hyperref}

\def\v2{v_2^{(2)}}

\def\diso{\lower.4ex\hbox{$\downarrow$}\raise.4ex\hbox{\mbox{\scriptsize $\wr$}}}

\def\iso{\,\lower .6ex\hbox{$\stackrel{\lra}{\mbox{\tiny $\sim\,$}}$}\,}

\def\lg{l\raise.6ex\hbox to.2em{\hss.\hss}l}
\def\lra{\longrightarrow}

\def\orb{\hbox to  .3em{$\backslash$}\backslash}

\def\lra{\longrightarrow}

  \def\beq{\begin{equation}}
  \def\eeq{\end{equation}}

\newtheorem{theorem}{Theorem}[section]
\newtheorem{lemma}[theorem]{Lemma}
\newtheorem{proposition}[theorem]{Proposition}

\theoremstyle{definition}
\newtheorem{definition}[theorem]{Definition}

\newtheorem{remark}[theorem]{Remark}
\numberwithin{equation}{section}

\begin{document}
\setcounter{page}{1}

\centerline{}

\centerline{}
\allowdisplaybreaks
\title[\hfilneg \hfil Solvability of some Stefan type problems]{ Solvability of some Stefan type problems}

\author[Y. Akdim, M. El Ansari and S. Lalaoui ]{Youssef AKDIM$^1$, Mohammed EL ANSARI$^2$ and  Soumia  LALAOUI RHALI$^3$}

\address{$^{1}$ Sidi Mohamed Ben Abdellah University, LAMA Laboratory, Faculty of Sciences D'har El Mahraz, P.O. Box 1796 Atlas Fez, Morocco.}
\email{\textcolor[rgb]{0.00,0.00,0.84}{$^1$$youssef.akdim@usmba.ac.ma$}}

\address{$^{2,3}$ Sidi Mohamed Ben Abdellah University, Laboratory of Engineering Sciences, Multidisciplinary Faculty of Taza,
                  P.O. Box 1223 Taza, Morocco.}
\email{\textcolor[rgb]{0.00,0.00,0.84}{$^2$$
mohammed.elansari1@usmba.ac.ma$}, $
\textcolor{blue}{soumia.lalaoui@usmba.ac.ma}$}


\date{Received: xxxxxx; Revised: yyyyyy; Accepted: zzzzzz
\newline  \indent $^{*}$ Corresponding author}

\begin{abstract}
In this paper, we interest on some class of Stefan type problems. We prove the existence and uniqueness of renormalized solution in anisotropic Sobolev spaces  with data belongs to $L^1- data,$ based on the properties of the renormalized  trunctions and the generalized monotonicity method
in the functional spaces.
\end{abstract}

\maketitle


\section{Introduction}

Anisotropic elliptic equations have received much attention in recent years (see for example, \cite{ref:17}, \cite{ref:18}, \cite{ref:19}, \cite{ref:20}, \cite{ref:21} and their references). Time dependent versions of these equations have been used as mathematical models to describe the spread of an epidemic disease, see \cite{ref:22}. such evolution models also arise in fluid dynamics when the media has different conductivities in the different directions (see \cite{ref:23}, \cite{ref:23}), and electrorheological fluids (see \cite{ref:17} for more details) as an important class of non-Newtonian fluids.\\

We are interested in the study of the behavior of solutions for a class of Stefan-type problems of form:

\begin{equation*}
 (E,f)
   \left\{
\begin{array}{l@{~}l@{~}l@{~}l}
\beta(u)-div(a(x,Du)+F(u))\ni f  & \text{in}   &\Omega,\\ 
u=0 &  \text{on} & \partial\Omega.
\end{array}
 \right.
\end{equation*}

with $\Omega $ is a bounded domain in $\mathbb{R}^{N}(N\geq 1)$ and  $\partial \Omega $
Lipschitz boundary if $N\geq 2$, a right-hand side $f$ which is assumed to belong to $L^{\infty }(\Omega
)$ or $L^{1}(\Omega )$ for $(E,f)$. Furthermore, $F:\mathbb{R}\rightarrow
\mathbb{R^{N}}$ is locally lipschitz continuous and $\beta :\mathbb{R}%
\rightarrow 2^{\mathbb{R}}$ is a set valued, maximal monotone mapping such
that $0\in \beta (0)$ and \newline
$a:\Omega \times \mathbb{R^{N}}\rightarrow \mathbb{R^{N}}$ is a
Carath\'{e}odory function satisfying the following assumptions: \newline

\noindent $(\mathbf{H}_{1})-$ Coerciveness: there exists a positive constant
$\lambda $ such that
\begin{equation*}
\sum_{i=1}^{N}a_{i}(x,\xi ).\xi _{i}\geq \lambda \sum_{i=1}^{N}|\xi
_{i}|^{p_{i}}
\end{equation*}
holds for all $\xi \in \mathbb{R}^{N}$ and almost every $x\in \Omega $.%
\newline

\noindent $(\mathbf{H}_{2})-$ Growth restriction:

\begin{center}
$|a_i(x,\xi)|\leq \gamma (d_i(x)+|\xi_i^{p_i-1}|) $
\end{center}

for almost every $x\in \Omega$, $\gamma$ is a positive constant for $%
i=1,...,N$, $d_i$ is a positive function in $L^{p^{\prime}_i}(\Omega)$ and
every $\xi \in \mathbb{R}^N$.\newline

\noindent $(\mathbf{H}_{3})-$ Monotonicity in $\xi \in \mathbb{R}^{N}$:
\begin{equation*}
(a(x,\xi )-a(x,\eta )).(\xi -\eta )\geq 0,
\end{equation*}
for almost every $x\in \Omega $ and for $\xi ,\eta \in \mathbb{R}^{N}$.%
\newline

Due to the possible jumps of $\beta$, problem $(E,f)$ enters to class of stefan problem for wich there exists a large number of references, among them \cite{ref:25}, \cite{ref:26}.\\
Here we use the notion of renormalised solution developed by by DiPerna and Lions \cite{D1}, for first order equations for $L^1-$data in \cite{ref:32}, and for Radon mesure data in \cite{ref:33}. It was then extended to the study of various problems of partial differential equations of parabolic, elliptic-parabolic and hyperbolic type, we refer to \cite{ref:27}, \cite{ref:28}.\\
Our problem has been studied in variable exponents spaces and Orlicz by Wittbold et al. \cite{ref:1}, \cite{ref:2} and in the weighted Sobolev spaces by Akdim and Allalou \cite{ref:5}. Other works in this direction can be found in \cite{ref:3},\cite{ref:11},\cite{ref:13}.
Our objective in this work is to prove an existence result of $(E,f)$ in
anisotropic Sobolev spaces, this notion were introduced by Nikolskii \cite{ref:29}, and Troisi \cite{ref:12}. The main tools in our proofs are Poincar\'{e}
inequality and the embedding for anisotropic Sobolev spaces. It would be interesting to refer to some Embedding theorems for anisotropic Sobolev-Orlicz spaces \cite{ref:30} and a fully anisotropic Sobolev Inequality
established by Cianchi in \cite{ref:31}.  \\
The paper isorganized as follows: In Section $2$, we recall the standard framework of
anisotropic Sobolev spaces and some notations which will be used frequently.
In Section $3$, we introduce the notion of weak and also renormalized
solution for the problem $(E,f)$ for any $L^{1}$-data. In Section $4,$ we
give our main results on the existence and uniqueness of renormalized
solutions and we discuss the existence of weak solutions. Section 5 is
devoted to the case where $f\in L^{\infty }(\Omega )$,\ we prove the
existence of a renormalized solution. Based on this result, the existence
and uniqueness of a\ renormalized solution in the case where $f\in L^{1}(\Omega )$ is shown in Section 6. In Section 7, we will prove the
existence of a weak solution. Finally, we give an example for illustrating our abstract result. 

\section{ Function spaces}

\subsection{Anisotropic Sobolev spaces}

Let $\Omega $ be a bounded open subset of $\mathbb{R}^{N}$, $(N\geq 2)$ and
let $1\leq p_{1},...,p_{N}<\infty $ be $N$ real numbers, $%
p^{+}=max(p_{1},...,p_{N})$, $p^{-}=min(p_{1},...,p_{N})$ and $%
\overrightarrow{p}=(p_{1},...,p_{N})$. The anisotropic spaces (see \cite{ref:12})

\begin{center}
$W^{1,{\overrightarrow{p}}}(\Omega)=\{u\in W^{1,1}(\Omega):\partial_{x_i}u
\in L^{p_i}(\Omega), i=1,...,N \}$.
\end{center}

is a Banach space with respect to norm

\begin{center}
\begin{equation*}
\lVert u\lVert_{W^{1,{\overrightarrow{p}}}(\Omega)}=\lVert
u\lVert_{L^1(\Omega)} + \sum_{i=1}^{N}\lVert
\partial_{x_i}u\lVert_{L^{p_i}(\Omega)}
\end{equation*}
.
\end{center}

The space $W_{0}^{1,{\overrightarrow{p}}}(\Omega )$ is the closure of $%
C_{0}^{\infty }(\Omega )$ with respect to this norm.\newline
The dual space of anisotropic Sobolev space $W_{0}^{1,{\overrightarrow{p}}%
}(\Omega )$ is equivalent to $W^{-1,{\overrightarrow{p^{\prime }}}}(\Omega )$%
, where $\overrightarrow{p^{\prime }}=(p_{1}^{\prime },...,p_{N}^{\prime })$
and $p_{i}^{\prime }=\dfrac{p_{i}}{p_{i}-1}$ for all $i=1,...,,N$.

We recall now a Poincar\'{e}-type inequality: \newline
Let $u\in W_{0}^{1,{\overrightarrow{p}}}(\Omega )$, then for every $q\geq 1$
there exists a constant $C_{p}$ (depending on $q$ and $i$ (see \cite{ref:16}), such that
\begin{equation}\label{eq2.1}
\lVert u\lVert _{L^{q}(\Omega )}\leq C_{p}\lVert \partial _{x_{i}}u\lVert
_{L^{p_{i}}(\Omega )}\hspace*{2mm}\text{for}\hspace*{2mm}i=1,...,N.
\end{equation}
Moreover a Sobolev-type inequality holds. Let us denote by $\overline{p}$
the harmonic mean of these numbers, i.e. $\dfrac{1}{\overline{p}}=\dfrac{1}{N%
}\displaystyle\sum\limits_{i=1}^{N}\frac{1}{p_{i}}$. Let $u\in W_{0}^{1,{%
\overrightarrow{p}}}(\Omega )$. It follows from \cite{ref:12} that there
exists a constant $C_{s}$ such that
\begin{equation}\label{eq2.2}
\lVert u\lVert _{L^{q}(\Omega )}\leq C_{s}\prod_{i=1}^{N}\lVert \partial
_{x_{i}}u\lVert _{L^{p_{i}}(\Omega )}^{\frac{1}{N}},
\end{equation}
where $q=\overline{p}^{\ast }=\frac{N\overline{p}}{N-\overline{p}}$ if $%
\overline{p}<N$ or $q\in \lbrack 1,+\infty \lbrack $ if $\overline{p}\geq N.$
On the right-hand side of $(\ref{eq2.2})$ it is possible to replace the geometric
mean by the arithmetic mean: let $a_{1},...,a_{N}$ be positive numbers, it
holds
\begin{equation*}
\displaystyle\prod_{i=1}^{N}a_{i}^{\frac{^{1}}{N}}\leq \dfrac{1}{N}%
\sum_{i=1}^{N}a_{i},
\end{equation*}
which implies by $\ref{eq2.2}$ that
\begin{equation}\label{eq2.3}
\lVert u\lVert _{L^{q}(\Omega )}\leq \dfrac{C_{s}}{N}\sum_{i=1}^{N}\lVert
\partial _{x_{i}}u\lVert _{L^{p_{i}}(\Omega )}.
\end{equation}
Note that when the following inequality holds
\begin{equation}\label{eq2.4}
\overline{p}<N,
\end{equation}
inequality $(\ref{eq2.3})$ implies the continuous embedding of the space $W_{0}^{1,{%
\overrightarrow{p}}}(\Omega )$ into $L^{q}(\Omega )$ for every $q\in \lbrack
1,\overline{p}^{\ast }].$ On the other hand, the continuity of the embeding $%
W_{0}^{1,{\overrightarrow{p}}}(\Omega )\hookrightarrow L^{p^{+}}(\Omega )$
with $p^{+}:=max\{p_{1},...,p_{N}\}$ relies on inequality $\ref{eq2.1}$. \newline
It may happen that $\overline{p}^{\ast }<p^{+}$ if the exponents $p_{i}$ are
closed enough, then \newline
$p_{\infty }:=max\{\overline{p}^{\ast },p^{+}\}$ turns out to be the
critical exponent in the anisotrpic Sobolev embedding (see \cite{ref:12}).

\begin{proposition}\label{2.1}
If the condition $\ref{eq2.4}$ holds, then for $q\in \lbrack 1,p_{\infty }]$ there
is a continuous embedding $W_{0}^{1,{\overrightarrow{p}}}(\Omega
)\hookrightarrow L^{q}(\Omega )$. For $q<p_{\infty }$ the embedding is
compact.
\begin{equation}\label{eq2.5}
W_{0}^{1,{\overrightarrow{p}}}(\Omega )\hookrightarrow \hookrightarrow
L^{q}(\Omega ).
\end{equation}
\end{proposition}

\subsection{Notations and functions}

Before we discuss the concept of solution we introduce some notations and
functions that will be frequently used.\newline
We begin by introducing the truncature operator. For given constant $k>0$ we
define the cut function $T_{k}:\mathbb{R}\rightarrow \mathbb{R}$ as
\begin{equation*}
T_{k}\left( r\right) =\left\{
\begin{array}{r@{~}c@{~}l}
-k,\text{if} & r\leq -k, &  \\
r,\text{if} & |r|<k, &  \\
k,\text{if} & r\geq k, &
\end{array}
\right.
\end{equation*}

\begin{figure}[h!]
\centering
\includegraphics[width=.7\textwidth]{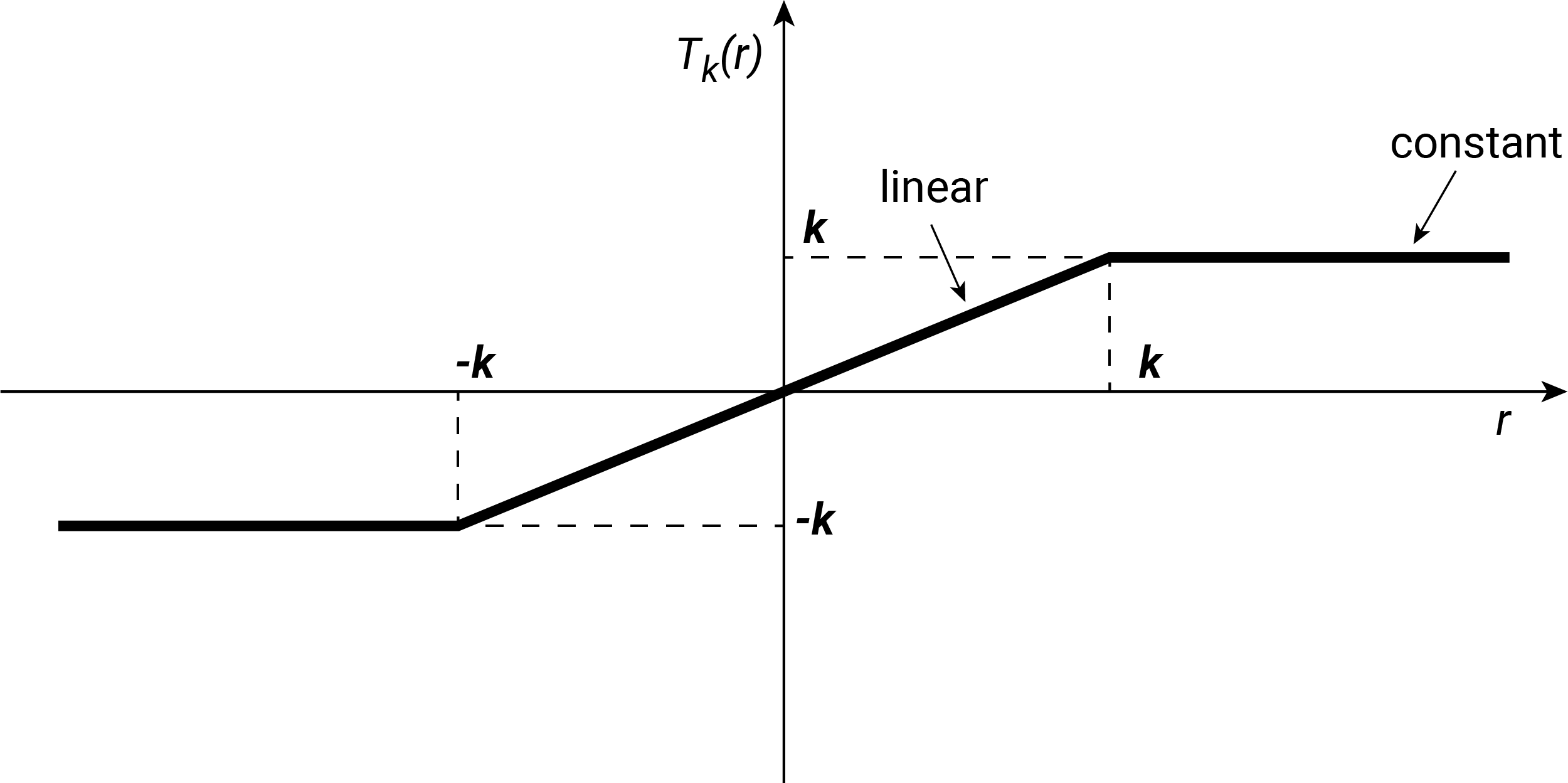}
\caption{Trunction function}
\end{figure}

and for $r\in \mathbb{R},$ let us define the functions : $r\rightarrow
r^{+}:=max(r,0)$ and $r\rightarrow sign_{0}(r)$ the usual sign function
which is defined by
\begin{equation*}
r\rightarrow sign_{0}(r):=\left\{
\begin{array}{r@{~}c@{~}l}
-1, & \text{on} & ]-\infty ,0[, \\
1, & \text{on} & ]0,\infty \lbrack , \\
0, & \text{if} & r=0.
\end{array}
\right.
\end{equation*}
and
\begin{equation*}
r\rightarrow sign_{0}^{+}(r):=\left\{
\begin{array}{r@{~}c@{~}l}
1, & \text{if} & r>0, \\
0, & \text{if} & r\leq 0.
\end{array}
\right.
\end{equation*}
Let $h_{l}:\mathbb{R}\rightarrow \mathbb{R}$ be defined by $%
h_{l}(r):=min((l+1-|r|)^{+},1)$ for each $r\in \mathbb{R}$.

\begin{figure}[h!]
\centering
\includegraphics[width=.8\textwidth]{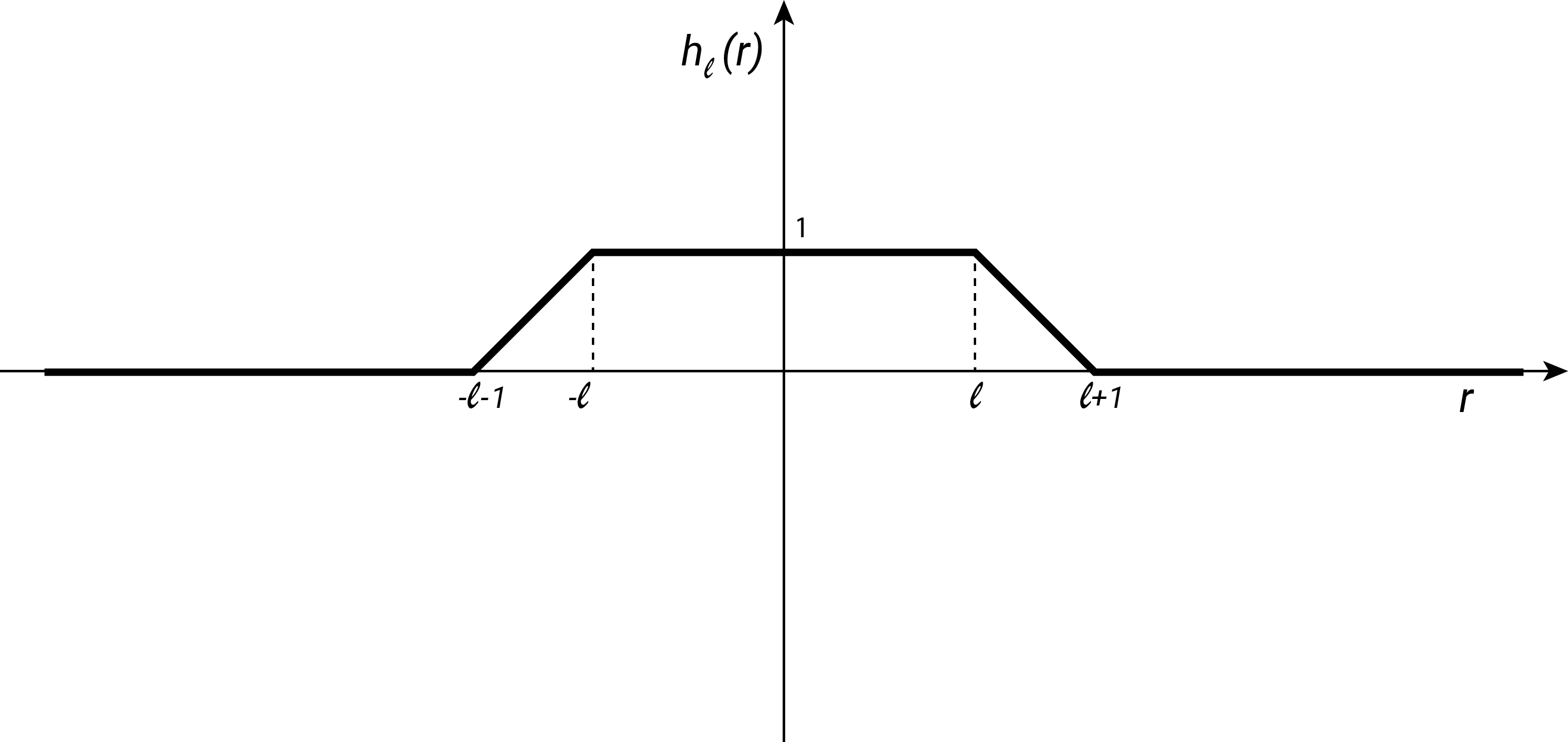}
\caption{Function $h_l(r)$}
\end{figure}

For $\sigma >0,$ we define $H_{\sigma }^{+}:\mathbb{R}\rightarrow \mathbb{R}$
by
\begin{equation*}
H_{\sigma }^{+}(r):=\left\{
\begin{array}{l@{~}l}
0,\text{if} & r<0, \\
\dfrac{1}{\sigma }r,\text{if} & 0\leq r\leq \sigma , \\
1,\text{if} & r>\sigma .
\end{array}
\right.
\end{equation*}
and $H_{\sigma }:\mathbb{R}\rightarrow \mathbb{R}$ by
\begin{equation*}
H_{\sigma }(r):=\left\{
\begin{array}{l@{~}l}
-1,\text{if} & r<-\sigma , \\
\dfrac{1}{\sigma }r,\text{if} & -\sigma \leq r\leq \sigma , \\
1,\text{if} & r>\sigma .
\end{array}
\right.
\end{equation*}

\section{Notion of solutions}

\subsection{Weak solutions}

\begin{definition}
A weak solution of $(E,f)$ is a pair of functions $(u,b)\in W_{0}^{1,{%
\overrightarrow{p}}}(\Omega )\times L^{1}(\Omega )$ satsfaying $F(u)\in
(L_{loc}^{1}(\Omega ))^{N},b\in \beta (u)$ almost everywhere in $\Omega $
and
\begin{equation}\label{eq3.1}
b-div(a(x,Du)+F(u))=f\hspace*{2mm}in\hspace*{2mm}D^{\prime }(\Omega ).
\end{equation}
\end{definition}

\subsection{Renormalized solutions}

\begin{definition}
A renormalized solution of $(E,f)$ is a pair of functions $(u,b)$ satisfying
the following conditions:\newline

$(\mathbf{R}_{1})$ $u:\Omega \rightarrow \mathbb{R}$ is measurable, $b\in
L^{1}(\Omega ),u(x)\in D(\beta (x))$ and $b(x)\in \beta (u(x))$ for a.e. $%
x\in \Omega $.\newline

$(\mathbf{R}_{2})$ For each $k>0,T_{k}(u)\in W_{0}^{1,{\overrightarrow{p}}%
}(\Omega )$ and
\begin{equation}\label{eq3.2}
\int_{\Omega }b.h(u)\phi +\int_{\Omega }(a(x,D(u))+F(u)).D(h(u)\phi
)=\int_{\Omega }fh(u)\phi ,
\end{equation}
holds for all $h\in C_{c}^{1}(\mathbb{R})$ and all $\phi \in W_{0}^{1,{%
\overrightarrow{p}}}(\Omega )\cap L^{\infty }(\Omega )$.\newline

$(\mathbf{R}_{3})$ $\int_{\{k<|u|<k+1\}}a(x,Du).Du\longrightarrow 0$ as $%
k\longrightarrow \infty $.
\end{definition}

\section{Main results}

In this section, we will first state the existence and uniqueness of
renormalized solutions for $(E,f).$ Then, we will prove that the
renormalized solution of $(E,f)$ is a weak solution.

\begin{theorem}\label{4.1}
 For $f\in L^{1}(\Omega )$, there exists at least one
renormalized solution $(u,b)$ of $(E,f)$.
\end{theorem}

\begin{theorem}\label{4.2}
 Let $\beta :\Omega \times \mathbb{R}\rightarrow 2^{\mathbb{R}}$
be such that $\beta (x,.)$ is strictly monotone for almost every $x\in
\Omega $. For $f\in L^{1}(\Omega ),$ let $(u,b),(\widetilde{u},\widetilde{b})
$ be renormalized solutions of $(E,f)$. Then $u=\widetilde{u}$ and $b=%
\widetilde{b}.$
\end{theorem}

\begin{proposition}\label{4.3}
 Let $(u,b)$ be a renormalized solution of $(E,f)$ for $f\in
L^{\infty }(\Omega )$. Then $u\in W^{1,{\overrightarrow{p}}}(\Omega )\cup
L^{\infty }(\Omega )$ and thus, in particular, $u$ is a weak solution of $%
(E,f)$.
\end{proposition}

To prove Theorem \ref{4.1}, we will introduce and solve approximation
problems. To this end, for $f\in L^{1}(\Omega )$ and $m,n\in \mathbb{N}$ we
define $f_{m,n}:\Omega \rightarrow \mathbb{R}$ by
\begin{equation*}
f_{m,n}(x)=\displaystyle max(min(f(x),m),-n)
\end{equation*}
for almost every $x\in \Omega $. Clearly, $f_{m,n}\in L^{\infty }(\Omega )$
for each $m,n\in \mathbb{N},$ $|f_{m,n}(x)|\leq |f(x)|$ a,e. in $\Omega $,
hence $\displaystyle\lim_{n\rightarrow \infty }\lim_{m\rightarrow \infty
}f_{m,n}=f$ in $L^{1}(\Omega )$ and almost everywhere in $\Omega $. The next
theorem will give us existence of renormalized solutions $(u_{m,n},b_{m,n})$
of $(E,f_{m,n})$ for each $m,n\in \Omega $.

\section{Case where $f\in L^{\infty }(\Omega )$}

\begin{theorem}\label{5.1}
\label{5.1} For $f\in L^{\infty }(\Omega )$, there exists at least one
renormalized solution $(u,b)$ of $(E,f)$.
\end{theorem}

The following section will be devoted to prove of Theorem \ref{5.1}, and we
will divide it into several steps.

\subsection{Approximate solution for $L^\infty$- data}

First we will introduce the approximate problem to $(E,f)$ for $f\in
L^{\infty }(\Omega )$ and for which the existence can be proved by standard
variational arguments. For $0<\varepsilon \leq 1,$ let $\beta _{\varepsilon
}:\mathbb{R}\longmapsto \mathbb{R}$ be the Yosida approximation of $\beta $
(see \cite{ref:8}). We introduce the operators

\begin{equation*}
\mathit{A}_{1,\varepsilon }:W_{0}^{1,{\overrightarrow{p}}}(\Omega
)\rightarrow W^{-1,{\overrightarrow{p^{\prime }}}}(\Omega ),
\end{equation*}
\begin{equation*}
u\rightarrow \beta _{\varepsilon }(T_{1/\varepsilon }(u))+\varepsilon
\arctan (u)-diva(x,Du)
\end{equation*}
and
\begin{equation*}
\mathit{A}_{2,\varepsilon }:W_{0}^{1,{\overrightarrow{p}}}(\Omega
)\rightarrow W^{-1,{\overrightarrow{p^{\prime }}}}(\Omega ),
\end{equation*}
\begin{equation*}
u\rightarrow -divF(T_{1/\varepsilon }(u)).
\end{equation*}
Because of $(\mathbf{H}_{2})-(\mathbf{H}_{3})$ , $\mathit{A}_{1/\varepsilon
} $ is well-defined and monotone (see \cite{ref:15} for instance).\newline
Since $\beta _{\varepsilon }\circ T_{1/\varepsilon }$ is bounded and
continuous and thanks to the growth condition $(H_{2})$ on $a$, it follows
that $\mathit{A}_{1,\varepsilon }$ is hermicontinuous (see \cite{ref:15}).
From the continuity and boundedness of $F\circ T_{1/\varepsilon },$ it
follows that $\mathit{A}_{2,\varepsilon }$ is strongly continuous. Therefore
the operator $\mathit{A}_{\varepsilon }:=\mathit{A}_{1,\varepsilon }+\mathit{%
A}_{2,\varepsilon }$ is pseudomonotone. Using the monotonicity of $\beta
_{\varepsilon }$, the Gauss-Green Theorem for Sobolev functions and the
boundary condition on the convection term $\int_{\Omega }F(T_{1/\varepsilon
}(u)).Du$, we show by using similar arguments as in \cite{ref:4} that $%
\mathit{A}_{\varepsilon }$ is coercive and bounded. Then it follows from \cite{ref:15} Theorem $2.7$ that $\mathit{A}_{\varepsilon }$ is surjective, i.e., for each
$0<\varepsilon \leq 1$ and $f\in W^{-1,{\overrightarrow{p^{\prime }}}%
}(\Omega )$ there exists a solution $u_{\varepsilon }\in W_{0}^{1,{%
\overrightarrow{p}}}(\Omega )$ of the problem
\begin{equation*}
(E_{\varepsilon },f)\left\{
\begin{array}{l@{~}l}
\displaystyle\beta _{\varepsilon }(T_{1/\varepsilon }(u_{\varepsilon
}))+\varepsilon \arctan (u_{\varepsilon })-div(a(x,Du_{\varepsilon
})+F((T_{1/\varepsilon }(u_{\varepsilon }))))=f & \text{in }\Omega , \\
u_{\varepsilon }=0 & \text{on }\partial \Omega ,
\end{array}
\right.
\end{equation*}

such that the following inequality holds for all $\phi \in W_{0}^{1,{%
\overrightarrow{p}}}(\Omega )$%
\begin{equation}\label{eq5.1}
\int_{\Omega }(\beta _{\varepsilon }(T_{1/\varepsilon }(u_{\varepsilon
}))+\varepsilon \arctan (u))\phi +\int_{\Omega }(a(x,D(u_{\varepsilon
}))+F(T_{1/\varepsilon }(u_{\varepsilon }))).D\phi =<f,\phi>
\end{equation}
where $<.,.>$ denotes the duality pairing between $W_{0}^{1,{\overrightarrow{%
p}}}(\Omega )$ and $W^{-1,{\overrightarrow{p^{\prime }}}}(\Omega ).$

\begin{proposition}\label{5.2}
For $0<\varepsilon \leq 1$ fixed and $f,\widetilde{f}\in L^{\infty }(\Omega )
$, let $u_{\varepsilon },\widetilde{u}_{\varepsilon }\in W_{0}^{1,{%
\overrightarrow{p}}}(\Omega )$ be solutions of $(E_{\varepsilon },f)$ and $%
(E_{\varepsilon },\widetilde{f})$, respectively, then, the follwing
comparison principle holds:

\begin{equation}\label{eq5.2}
\varepsilon \int_{\Omega }(\arctan (u_{\varepsilon })-\arctan (\widetilde{u}%
_{\varepsilon }))^{+}\leq \int_{\Omega }(f-\widetilde{f})sign_{0}^{+}(u_{%
\varepsilon }-\widetilde{u}_{\varepsilon }).
\end{equation}
\end{proposition}

\begin{proof}
We use the test function $\varphi =H_{\sigma }^{+}(u_{\varepsilon }-%
\widetilde{u}_{\varepsilon })$ in the weak formulation $(\ref{eq5.1})$ for $%
u_{\varepsilon }$ and $\widetilde{u}_{\varepsilon }$. Substracting the
resulting inequalties, we obtain
\begin{equation*}
I_{l,\delta }^{1}+I_{l,\delta }^{2}+I_{l,\delta }^{3}+I_{l,\delta
}^{4}=I_{l,\delta }^{5},
\end{equation*}
where

\begin{equation*}
I_{l,\delta }^{1}=\int_{\Omega }(\beta _{\varepsilon }T_{\frac{1}{%
\varepsilon }}(u_{\varepsilon })-\beta _{\varepsilon }(T_{\frac{1}{%
\varepsilon }}(\widetilde{u}_{\varepsilon })))H_{\delta }^{+}(u_{\varepsilon
}-\widetilde{u}_{\varepsilon })
\end{equation*}
\begin{equation*}
I_{l,\delta }^{2}=\int_{\Omega }(\varepsilon \arctan (u_{\varepsilon
})-\varepsilon \arctan (\widetilde{u}_{\varepsilon }))H_{\delta
}^{+}(u_{\varepsilon }-\widetilde{u}_{\varepsilon }),
\end{equation*}
\begin{equation*}
I_{l,\delta }^{3}=\int_{\Omega }a(x,Du_{\varepsilon })-a(x,D\widetilde{u}%
_{\varepsilon }).DH_{\delta }^{+}(u_{\varepsilon }-\widetilde{u}%
_{\varepsilon }),
\end{equation*}
\begin{equation*}
I_{l,\delta }^{4}=\int_{\Omega }(F(T_{\frac{1}{\varepsilon }}(u_{\varepsilon
}))-F(T_{\frac{1}{\varepsilon }}(\widetilde{u}_{\varepsilon }))).DH_{\delta
}^{+}(u_{\varepsilon }-\widetilde{u}_{\varepsilon })
\end{equation*}

\begin{equation*}
I_{l,\delta }^{5}=\int_{\Omega }(f-\widetilde{f})H_{\delta
}^{+}(u_{\varepsilon }-\widetilde{u}_{\varepsilon }).
\end{equation*}
Passing to the limit with $\sigma \rightarrow 0$, $(\ref{eq5.2})$ follows.
\end{proof}

\begin{remark}\label{5.3}
Let $f,\widetilde{f}\in L^{\infty }(\Omega )$ be such that $f\leq \widetilde{%
f}$ almost everywhere in $\Omega ,$ $\varepsilon >0$ and $u_{\varepsilon },$
$\widetilde{u}_{\varepsilon }\in W_{0}^{1,{\overrightarrow{p}}}(\Omega )$ be
solutions of $(E_{\varepsilon },f)$ and $(E_{\varepsilon },\widetilde{f})$,
respectively, then it is an immediate consequence of Propsition \ref{5.2} is
that $u_{\varepsilon }\leq \widetilde{u}_{\varepsilon }$ almost everywhere
in $\Omega $. Furthermore, from the monotonocity of $\beta _{\varepsilon
}\circ T_{1/\varepsilon }$ it follows that also \newline
$$\beta _{\varepsilon }(T_{1/\varepsilon }(u_{\varepsilon }))\leq \beta
_{\varepsilon }(T_{1/\varepsilon }(\widetilde{u}_{\varepsilon }))$$ a,e. in $%
\Omega $.
\end{remark}

\subsection{A priori estimates}

\begin{lemma}\label{5.4}
For $0<\varepsilon \leq 1$ and $f\in L^{\infty }(\Omega )$ let $%
u_{\varepsilon }\in W_{0}^{1,{\overrightarrow{p}}}(\Omega ))$ be a solution
of $(E_{\varepsilon },f)$. Then \newline
$i)$ There exists a constant $C_{1}=C_{1}\displaystyle(||f||_{\infty
},\lambda ,p_i,N)>0,$ not depending on $\varepsilon $, such that
\begin{equation}\label{eq5.3}
|||u_{\varepsilon }|||\leq C_{1}.
\end{equation}
$ii)$ for all $0<\varepsilon \leq 1,$ we have
\begin{equation}\label{eq5.4}
||\beta _{\varepsilon }(T_{1/\varepsilon }(u_{\varepsilon }))||_{\infty
}\leq ||f||_{\infty }
\end{equation}
$iii)$ for all $0<\varepsilon \leq 1$ and all $l,k>0,$ we have
\begin{equation}\label{eq5.5}
\int_{\{l\leq |u|\leq k+l\}}a(x,Du_{\varepsilon }).Du_{\varepsilon }\leq
k\int_{\{|u_{\varepsilon }|>l\}}|f|.
\end{equation}
\end{lemma}

\begin{proof}
$i)$ Taking $u_{\varepsilon }$ as a test function in $(\ref{eq5.1})$ we obtain
\begin{equation*}
\int_{\Omega }(\beta _{\varepsilon }(T_{1/\varepsilon }(u_{\varepsilon
}))+\varepsilon \arctan (u_{\varepsilon }))u_{\varepsilon } dx +\int_{\Omega
}a(x,Du_{\varepsilon }).Du_{\varepsilon } dx +\int_{\Omega }F(T_{1/\varepsilon }(u_{\varepsilon
})).Du_{\varepsilon } dx =\int_{\Omega }fu_{\varepsilon }dx
\end{equation*}
As the first term on the left-hand side is nonnegative and the integral over
the convextion term vanishes by $(H_{1})$, we have
\begin{equation*}
\displaystyle\lambda \sum_{i=1}^{N}\int_{\Omega }|\dfrac{\partial
u_{\varepsilon }}{\partial x_{i}}|^{p_{i}}dx\leq \sum_{i=1}^{N}\int_{\Omega
}a_{i}(x,Du_{\varepsilon }).\dfrac{\partial u_{\varepsilon }}{\partial x_{i}}%
dx\leq \int_{\Omega }fu_{\varepsilon }dx\leq C||f||_{\infty
}(\sum_{i=1}^{N}\int_{\Omega }|\dfrac{\partial u_{\varepsilon }}{\partial
x_{i}}|^{p_{i}}dx)^{1/p_{i}}.
\end{equation*}
due to the H\"{o}lder inequality. Thus $|||u_{\varepsilon }|||^{p_{i}}\leq
C_{2}|||u_{\varepsilon }|||$, where $C_{2}$ is a positive constant. Then we
can deduce that $u_{\varepsilon }$ remains bounded in $W_{0}^{1,{%
\overrightarrow{p}}}(\Omega )$ i.e.,
\begin{equation*}
|||u_{\varepsilon }|||\leq C_{1}.
\end{equation*}
$ii)$ Taking $\dfrac{1}{\sigma }[T_{k+\sigma }(\beta _{\varepsilon
}(T_{1/\varepsilon }(u_{\varepsilon })))-T_{k}(\beta _{\varepsilon
}(T_{1/\varepsilon }(u_{\varepsilon })))]$ as a test function in $(\ref{eq5.1})$,
passing to the limit with $\sigma \rightarrow 0$ and choosing $%
k>||f||_{\infty }$, we obtain $ii)$.\newline
$iii)$ For $k,l>0$ fixed we take $T_{k}(u_{\varepsilon
}-T_{l}(u_{\varepsilon }))$ as a test function in $(\ref{eq5.1})$.\newline
Using $\int_{\Omega }a(x,Du_{\varepsilon }).DT_{k}(u_{\varepsilon
}-T(u_{\varepsilon }))dx=\int_{\{l<|u_{\varepsilon
}|<l+k\}}a(x,Du_{\varepsilon }).Du_{\varepsilon }dx$, and as the first term
on the left-hand side is nonnegative and the convection term vanishes, we
get
\begin{equation}\label{eq5.6}
\int_{\{l<|u_{\varepsilon }|<k+l\}}a(x,Du_{\varepsilon }).Du_{\varepsilon
}\leq \int_{\Omega }fT_{k}(u_{\varepsilon }-T_{l}(u_{\varepsilon }))dx\leq
\int_{\{|u_{\varepsilon }|>l\}}|f|dx.
\end{equation}
\end{proof}

\begin{remark}\label{5.5}
For $k>0$, from $iii)$ in Lemma \ref{5.4}, we deduce that
\begin{equation}\label{eq5.7}
\displaystyle|\{|u_{\varepsilon }|\geq l\}|\leq \dfrac{C_{2}}{l^{1-\frac{1}{%
\overline{p}}}}
\end{equation}
\begin{equation}\label{eq5.8}
\int_{\{l\leq |u_{\varepsilon }|\leq k+l\}}a(x,Du_{\varepsilon
}).Du_{\varepsilon }\leq k||f||_{\infty }|\{|u_{\varepsilon }|>l\}|\leq
\dfrac{C_{2}(k)}{l^{\frac{1}{\overline{p}}-1}}
\end{equation}
for any $0<\varepsilon \leq 1$ and a constant $C_{2}(k)>0$ not depending on $%
\varepsilon $.
\end{remark}
Indeed, let $l>0$ large enough we have:\\
\begin{equation*}
l \lvert\{|u_\varepsilon|\geq l\}\rvert= \int_{\{|u_\varepsilon|\geq l\}}\lvert T_l(u_\varepsilon)\rvert dx \leq C \Bigg(\sum_{i=1}^{N}\int_{\Omega}\Big\lvert\dfrac{\partial T_l(u_\varepsilon)}{\partial x_i} \Big\rvert^{p_i}\Bigg)^{1/p_i} \leq C_2l^{1/\overline{p}}
\end{equation*}
with implies $ \lvert\{|u_\varepsilon|\geq l\}\rvert \leq  C_2l^{1/\overline{p}-1}$. Then 
$$\lim_{l\rightarrow +\infty} \lvert\{|u_\varepsilon|\geq l\}\rvert=0.$$
Therefore, (\ref{eq5.7}) follows from (\ref{eq5.8}).

\subsection{Basic convergence results}

\begin{lemma}\label{5.6}
\label{5.6}For $0<\varepsilon \leq 1$ and $f\in L^{\infty }(\Omega )$, let $%
u_{\varepsilon }\in W_{0}^{1,{\overrightarrow{p}}}(\Omega )$ be a solution
of $(E_{\varepsilon },f)$. There exist $u\in W_{0}^{1,{\overrightarrow{p}}%
}(\Omega ),b\in L^{\infty }(\Omega )$ such that for a not relabeled
subsequence of $(u_{\varepsilon })_{0<\varepsilon \leq 1}$ as $\varepsilon
\mapsto 0:$

\begin{equation}\label{eq5.9}
u_{\varepsilon }\rightharpoonup u\hspace*{1cm}in\hspace*{2mm}W_{0}^{1,{%
\overrightarrow{p}}}(\Omega
)\hspace*{2mm}and\hspace*{2mm}a.e.\hspace*{2mm}in\hspace*{2mm}\Omega ,
\end{equation}

\begin{equation}\label{eq5.10}
T_{k}(u_{\varepsilon })\rightharpoonup
T_{k}(u)\hspace*{0.5cm}in\hspace*{2mm}W_{0}^{1,{\overrightarrow{p}}}(\Omega
)\hspace*{2mm}and\hspace*{2mm}strongly\hspace*{2mm}in\hspace*{2mm}L^{q}(%
\Omega ),
\end{equation}

\begin{equation}\label{eq5.11}
\beta _{\varepsilon }(T_{1/\varepsilon }(u_{\varepsilon }))\rightharpoonup b%
\hspace{0.5cm}in\hspace*{2mm}L^{\infty }(\Omega ).
\end{equation}
Moreover, for any $k>0$,
\begin{equation}\label{eq5.12}
DT_{k}(u_{\varepsilon })\rightharpoonup
DT_{k}(u)\hspace*{0.5cm}in\hspace*{2mm}\prod_{i=1}^{N}L^{p_{i}}(\Omega ),
\end{equation}

\begin{equation}\label{eq5.13}
a(x,DT_{k}(u_{\varepsilon }))\rightharpoonup
a(x,DT_{k}(u))\hspace*{0.5cm}in\hspace*{2mm}\prod_{i=1}^{N}L^{p_{i}^{^{%
\prime }}}(\Omega ).
\end{equation}
\end{lemma}

\begin{proof}
By combining Lemma \ref{5.4} and Remark \ref{5.5}, we obtain $(\ref{eq5.11})$. From $(\ref{eq5.7}), (\ref{eq5.3})$
and $(\ref{eq2.5})$, we deduce with a classical argument (see \cite{ref:14}) that
for a subsequence still indexed by $\varepsilon $, $(\ref{eq5.9})-(\ref{eq5.10})$ and $(\ref{eq5.12})$
hold as $\varepsilon $ tends to $0$, where $u$ is a mesurable function
defined on $\Omega $.

It is left to prove $(\ref{eq5.13})$. For this, by $(H_{2})$ and $(\ref{eq5.3})$ it follows
that given any subsequence of $(a(x,DT_{k}(u_{\varepsilon }))_{\varepsilon })
$, there exists a subsequence, still denoted by $(a(x,DT_{k}(u_{\varepsilon
}))_{\varepsilon })$, such that $a(x,DT_{k}(u_{\varepsilon
}))\rightharpoonup \Phi _{k}$ in $\displaystyle\prod_{i=1}^{N}L^{p_{i}^{^{%
\prime }}}(\Omega )$. We will prove that $\Phi _{k}=a(x,DT_{k}(u))$ a.e. on $%
\Omega $. The proof consists of three steps.\newline
\underline{Step 1:} For every $h\in W^{1,\infty }(\mathbb{R}),h\leq 0$ and $%
supp(h)$ compact, we will prove that
\begin{equation}\label{eq5.14}
\limsup_{\varepsilon \rightarrow 0}\int_{\Omega }a(x,DT_{k}(u_{\varepsilon
})).D[h(u_{\varepsilon })(T_{k}(u_{\varepsilon })-T_{k}(u))]dx\leq 0.
\end{equation}
Taking $h(u_{\varepsilon })(T_{k}(u_{\varepsilon })-T_{k}(u))$ as a test
function in $(\ref{eq5.1})$, we have
\begin{equation}\label{eq5.15}
\begin{array}{c}
\int_{\Omega }(\beta _{\varepsilon }(T_{1/\varepsilon }(u_{\varepsilon
})+\varepsilon \arctan (u_{\epsilon }))h(u_{\varepsilon
})(T_{k}(u_{\varepsilon })-T_{k}(u))+\int_{\Omega }a(x,D(u_{\varepsilon
})).D[h_{\varepsilon }(T_{k}(u_{\varepsilon })-T_{k}(u))] \\
+\int_{\Omega }F(T_{1/\varepsilon }(u_{\varepsilon })).D[h_{\varepsilon
}(T_{k}(u_{\varepsilon })-T_{k}(u))]=\int_{\Omega }fh(u_{\varepsilon
})(T_{k}(u_{\varepsilon })-T_{k}(u)).
\end{array}
\end{equation}

Using $|h_{\varepsilon }(T_{k}(u_{\varepsilon })-T_{k}(u))|\leq
2k||h||_{\infty }$, by Lebesgue's dominated convergence theorem we find that
$\lim\limits_{\varepsilon \rightarrow 0}\int_{\Omega }fh(u_{\varepsilon
})(T_{k}(u_{\varepsilon })-T_{k}(u))=0$ and then \newline
$\lim\limits_{\varepsilon \rightarrow 0}\int_{\Omega }F(T_{1/\varepsilon
}(u_{\varepsilon })).D[h_{\varepsilon }(u_{\varepsilon
})(T_{k}(u_{\varepsilon })-T_{k}(u))]=0$. By using the same arguments as in \cite{ref:3}, we can prove that
\begin{equation*}
\limsup_{\varepsilon \rightarrow 0}\int_{\Omega }\beta _{\varepsilon
}(T_{1/\varepsilon }(u_{\varepsilon })).[h(u_{\varepsilon
})(T_{k}(u_{\varepsilon })-T_{k}(u))]dx\geq 0.
\end{equation*}
Passing to the limit in $(\ref{eq5.15})$ and using the above results, we obtain $%
(\ref{eq5.14})$.

\underline{Step 2:} We now prove that for every $k>0$,
\begin{equation}\label{eq5.16}
\limsup_{\varepsilon \rightarrow 0}\int_{\Omega }a(x,DT_{k}(u_{\varepsilon
})).[D(T_{k}(u_{\varepsilon })-DT_{k}(u))]dx\leq 0.
\end{equation}
Indeed, for $k>l$, take $h_{l}(u_{\varepsilon })(T_{k}(u_{\varepsilon
})-T_{k}(u))$ as a test function in $(\ref{eq5.1})$. Letting $\varepsilon
\rightarrow 0$ and then $l\rightarrow \infty $, we obtain
\begin{equation*}
\limsup_{\varepsilon \rightarrow 0}\int_{\Omega }a(x,DT_{k}(u_{\varepsilon
})).D[h_{l}(u_{\varepsilon })(T_{k}(u_{\varepsilon
})-T_{k}(u))]dx=E_{1}+E_{2}+E_{3}.
\end{equation*}
where
\begin{equation*}
\begin{array}{l}
E_{1}=\int_{\{|u_{\varepsilon }|\leq k\}}h_{l}(u_{\varepsilon
})a(x,DT_{k}(u_{\varepsilon })).[DT_{k}(u_{\varepsilon })-DT_{k}(u)]dx, \\
\\
E_{2}=\int_{\{|u_{\varepsilon }|>k\}}h_{l}(u_{\varepsilon
})a(x,DT_{k}(u_{\varepsilon })).(-DT_{k}(u))]dx, \\
\\
E_{3}=\int_{\Omega }h_{l}^{^{\prime }}(u_{\varepsilon
})(T_{k}(u_{\varepsilon })-T_{k}(u))a(x,DT_{k}(u_{\varepsilon
})).Du_{\varepsilon }dx.
\end{array}
\end{equation*}
Since $l>k$, on the set $\{|u_{\varepsilon }|\leq k\}$ we have $%
h_{l}(u_{\varepsilon })=1$ so that we can write
\begin{equation*}
\limsup_{\varepsilon \rightarrow 0}E_{1}=\limsup_{\varepsilon \rightarrow
0}\int_{\Omega }a(x,DT_{k}(u_{\varepsilon })).(DT_{k}(u_{\varepsilon
})-DT_{k}(u))dx.
\end{equation*}
For $E_{2}$, using Lebesgue's dominated convergence theorem, we get
\begin{equation*}
\lim_{\varepsilon \rightarrow 0}E_{2}=\int_{\{|u_{\varepsilon
}|>k\}}h_{l}(u)\Phi _{l+1}.DT_{k}(u)dx=0.
\end{equation*}
For $E_{3}$, we have
\begin{equation*}
-\int_{\Omega }h_{l}^{^{\prime }}(u_{\varepsilon })(T_{k}(u_{\varepsilon
})-T_{k}(u))a(x,DT_{k}(u_{\varepsilon }))Du_{\varepsilon }dx
\end{equation*}
\hspace*{4cm}
\begin{equation*}
\leq 2k\int_{\{l<|u_{\varepsilon }|\leq l+1\}}a(x,Du_{\varepsilon
})Du_{\varepsilon }dx.
\end{equation*}

Using $(\ref{eq5.8})$, we deduce that
\begin{equation*}
\limsup_{l\rightarrow \infty }\limsup_{\varepsilon \rightarrow
0}(-\int_{\Omega }h_{l}^{^{\prime }}(u_{\varepsilon })(T_{k}(u_{\varepsilon
})-T_{k}(u))a(x,DT_{k}(u_{\varepsilon })).Du_{\varepsilon }dx)\leq 0.
\end{equation*}
Applying $(\ref{eq5.14})$ with $h$ replaced by $h_{l},l>k$, we get
\begin{equation*}
\limsup_{\varepsilon \rightarrow 0}\int_{\Omega }a(x,DT_{k}(u_{\varepsilon
})).[DT_{k}(u_{\varepsilon })-DT_{k}(u)]dx
\end{equation*}
\begin{equation*}
\hspace*{4cm}\leq \limsup_{\varepsilon \rightarrow 0}(-\int_{\Omega
}h_{l}^{^{\prime }}(u_{\varepsilon })(T_{k}(u_{\varepsilon
})-T_{k}(u))a(x,DT_{k}(u_{\varepsilon })).Du_{\varepsilon }dx).
\end{equation*}

Now letting $l\rightarrow \infty ,$ $(\ref{eq5.16})$ yields.

\underline{Step 3:} In this step, we prove by monotonicity arguments that
for $k>0$, \newline
$\Phi _{k}=a(x,DT_{k}(u))$ for almost every $x\in \Omega $. Let $\phi \in
D(\Omega )$ and $\alpha \in \mathbb{R}$. Using $(\ref{eq5.16})$, we have
\begin{equation*}
\alpha \lim_{\varepsilon \rightarrow 0}\int_{\Omega
}a(x,DT_{k}(u_{\varepsilon })).D\phi dx\geq \alpha \int_{\Omega
}a(x,D(T_{k}(u)-\alpha \phi )).D\phi dx.
\end{equation*}
Dividing by $\alpha >0$ and $\alpha <0$ and letting $\alpha \rightarrow 0$,
we obtain
\begin{equation}\label{eq5.17}
\lim_{\varepsilon \rightarrow 0}\int_{\Omega }a(x,DT_{k}(u_{\varepsilon
})).D\phi dx=\int_{\Omega }a(x,DT_{k}(u)).D\phi dx.
\end{equation}
This means that for all $k>0$, $\displaystyle\int_{\Omega }\Phi _{k}.D\phi
dx=\int_{\Omega }a(x,DT_{k}(u)),$ and then \newline
$\Phi _{k}=a(x,DT_{k}(u))$ in $D^{\prime }(\Omega )$ for all $k>0$. Hence $%
\Phi _{k}=a(x,DT_{k}(u))$ a.e. in $\Omega $ and then $a(x,DT_{k}(u_{%
\varepsilon }))\rightharpoonup a(x,DT_{k}(u))$ weakly in $\displaystyle%
\prod_{i=1}^{N}L^{p_{i}^{^{\prime }}}(\Omega ).$
\end{proof}

\begin{remark}\label{5.7}
As an immediate consequence of $(\ref{eq5.16})$ and $(H_{3})$ we obtain
\begin{equation}\label{eq5.18}
\lim_{\varepsilon \rightarrow 0}\int_{\Omega }a(x,DT_{k}(u_{\varepsilon
})-a(x,DT_{k}(u)).(DT_{k}(u_{\varepsilon}) -T_{k}(u))=0.
\end{equation}
Let us see finally that
\begin{equation}\label{eq5.19}
\lim_{l\rightarrow \infty }\int_{l<|u|<l+1}a(x,DT(u)dx=0.
\end{equation}
\end{remark}

Endeed, for any $l\geq 0$ fixed we have\newline
\begin{equation*}
\int_{l<|u|<l+1}a(x,D(u_{\varepsilon }).D(u_{\varepsilon })dx=\int_{\Omega
}a(x,DT_{l+1}(u_{\varepsilon }).(DT_{l+1}(u_{\varepsilon
})-DT_{l}(u_{\varepsilon }))dx
\end{equation*}
\begin{equation*}
=\int_{\Omega }a(x,DT_{l+1}(u_{\varepsilon })).DT_{l+1}(u_{\varepsilon
})dx-\int_{\Omega }a(x,DT_{l}(u_{\varepsilon })).DT_{l}(u_{\varepsilon })dx.
\end{equation*}
By $(\ref{eq5.18})$ and passing to the limit as $\varepsilon \rightarrow 0$ for
fixed $l\geq 0$ we obtain
\begin{equation}\label{eq5.20}
\begin{array}{cc}
\lim_{\varepsilon \rightarrow 0}\int_{\{l<|u_\varepsilon|<l+1\}}a(x,D(u_{\varepsilon
})).D(u_{\varepsilon })dx  =\int_{\Omega
}a(x,DT_{l+1}(u)).DT_{l+1}(u)dx &\\  - \int_{\Omega
}a(x,DT_{l}(u)).DT_{l}(u)dx 
 =\int_{\{l<|u|<l+1\}}a(x,Du).D(u)dx.
\end{array}
\end{equation}
Therefore, taking $l\rightarrow +\infty $ in $(\ref{eq5.20})$ and using the estimate $(\ref{eq5.8})$ show that satisfies $(R_3)$.

\subsection{Proof of the existence result}

We are now in position to conclude the proof of our main result presented in
Theorem \ref{5.1}:

\begin{proof}
Let $h\in C_{c}^{1}(\mathbb{R})$ and $\varphi \in W_{0}^{1,{\overrightarrow{p%
}}}(\Omega )\cap L^{\infty }(\Omega )$. Taking $h_{l}(u_{\varepsilon
})h(u)(\varphi )$ as a test function in $(5.1)$, we obtain
\begin{equation}\label{eq5.21}
I_{\varepsilon ,l}^{1}+I_{\varepsilon ,l}^{2}+I_{\varepsilon
,l}^{3}+I_{\varepsilon ,l}^{4}=I_{\varepsilon ,l}^{5}
\end{equation}
where
\begin{equation*}
I_{\varepsilon ,l}^{1}=\int_{\Omega }\beta _{\varepsilon }(T_{1/\varepsilon
}(u_{\varepsilon }))h_{l}(u_{\varepsilon })h(u)\varphi ,
\end{equation*}

\begin{equation*}
I_{\varepsilon ,l}^{2}=\varepsilon \int_{\Omega }\arctan (u_{\varepsilon
})h_{l}(u_{\varepsilon })h(u)\varphi ,
\end{equation*}

\begin{equation*}
I_{\varepsilon ,l}^{3}=\int_{\Omega }a(x,Du_{\varepsilon
}).D(h_{l}(u_{\varepsilon })h(u)\varphi ),
\end{equation*}

\begin{equation*}
I_{\varepsilon ,l}^{4}=\int_{\Omega }F(T_{1/\varepsilon }(u_{\varepsilon
})).D(h_{l}(u_{\varepsilon })h(u)\varphi ),
\end{equation*}

\begin{equation*}
I_{\varepsilon ,l}^{5}=\int_{\Omega }fh_{l}(u_{\varepsilon })h(u)\varphi .
\end{equation*}

\underline{Step 1:} Letting $\varepsilon \rightarrow 0$ obviously, we have%
\newline
\begin{equation}\label{eq5.222}
\lim_{\varepsilon \rightarrow 0}I_{\varepsilon ,l}^{2}=0.
\end{equation}
Using the convergence results $(\ref{eq5.9}),((\ref{eq5.11})$ from Lemma \ref{5.6} we can
immediately calculate the following limits:

\begin{equation}\label{eq5.22}
\lim_{\varepsilon \rightarrow 0}I_{\varepsilon ,l}^{1}=\int_{\Omega
}bh_{l}(u)h(u)\varphi ,
\end{equation}

\begin{equation}\label{eq5.23}
\lim_{\varepsilon \rightarrow 0}I_{\varepsilon ,l}^{5}=\int_{\Omega
}fh_{l}(u)h(u)\varphi .
\end{equation}
We write $I_{\varepsilon ,l}^{3}=I_{\varepsilon ,l}^{3,1}+I_{\varepsilon
,l}^{3,2}$ where
\begin{equation*}
\displaystyle I_{\varepsilon ,l}^{3,1}=\int_{\Omega }h_{l}^{\prime
}(u_{\varepsilon })a(x,Du_{\varepsilon }).Du_{\varepsilon }h(u)\varphi ,%
\text{ }I_{\varepsilon ,l}^{3,2}=\int_{\Omega }h_{l}(u_{\varepsilon
})a(x,Du_{\varepsilon })D(h(u)\varphi )..
\end{equation*}
\newline
Using $(\ref{eq5.8})$, we get the estimate
\begin{equation}\label{eq5.24}
|\lim_{\varepsilon \rightarrow 0}I_{\varepsilon ,l}^{3,1}|\leq ||h||_{\infty
}||\varphi ||_{\infty }.C_{2}l^{-(1-1/\bar{p})}.
\end{equation}
By Lebesgue's dominated convergence theorem it follows that for any $i\in
\{1,...,N\}$, we have
\begin{equation*}
h_{l}(u_{\varepsilon })\dfrac{\partial }{\partial x_{i}}(h(u)\varphi
)\rightarrow h_{l}(u)\dfrac{\partial }{\partial x_{i}}(h(u)\varphi
)\hspace*{5mm}in\hspace*{2mm}L^{p_{i}}\hspace*{2mm}as\hspace*{2mm}%
\varepsilon \rightarrow 0.
\end{equation*}
Keeping in mind that $I_{\varepsilon ,l}^{3,2}=\int_{\Omega
}h_{l}(u_{\varepsilon })a(x,DT_{l+1}(u_{\varepsilon })).D(h(u)\varphi )$ and
by using $(\ref{eq5.13})$, we get
\begin{equation}\label{eq5.25}
\lim_{\varepsilon \rightarrow 0}I_{\varepsilon ,l}^{3,2}=\int_{\Omega
}h_{l}(u)a(x,DT_{l+1}(u)).D(h(u)\varphi ).
\end{equation}
Let us write $I_{\varepsilon ,l}^{4}=I_{\varepsilon ,l}^{4,1}+I_{\varepsilon
,l}^{4,2}$, where
\begin{equation*}
I_{\varepsilon ,l}^{4,1}=\int_{\Omega }h_{l}^{\prime }(u_{\varepsilon
})F(T_{1/\varepsilon }(u_{\varepsilon })).Du_{\varepsilon }h(u)\varphi ,
\end{equation*}

\begin{equation*}
I_{\varepsilon ,l}^{4,2}=\int_{\Omega }h_{l}(u_{\varepsilon
})F(T_{1/\varepsilon }(u_{\varepsilon })).D(h(u)\varphi ).
\end{equation*}

For any $l\in \mathbb{N}$, there exists $\varepsilon _{0}(l)$ such that for
all $\varepsilon <\varepsilon _{0}(l),$
\begin{equation}\label{eq5.26}
I_{\varepsilon ,l}^{4,1}=\int_{\Omega }h_{l}^{\prime
}(T_{l+1}(u_{\varepsilon }))F(T_{l+1}(u_{\varepsilon })).h(u)\varphi .
\end{equation}

\noindent Using the Gauss-Green Theorem for Sobolev functions in $(\ref{eq5.26})$,
we get for all $\varepsilon <\varepsilon _{0}(l),$
\begin{equation}\label{eq5.27}
I_{\varepsilon ,l}^{4,1}=-\int_{\Omega }\int_{0}^{T_{l+1}(u_{\varepsilon
})}h_{l}^{\prime }(r)F(r)dr.D(h(u)\varphi ).
\end{equation}

\noindent Now, using $(\ref{eq5.9})$ and the Gauss-Green Theorem, after letting $%
\varepsilon \rightarrow 0$, we get

\begin{equation}\label{eq5.28}
\lim_{\varepsilon \rightarrow }I_{\varepsilon ,l}^{4,1}=\int_{\Omega
}h_{l}^{\prime }(u)F(u).Duh(u)\varphi .
\end{equation}

\noindent Choosing $\varepsilon $ small enough, we can write

\begin{equation}\label{eq5.29}
I_{\varepsilon ,l}^{4,2}=\int_{\Omega }h_{l}(u_{\varepsilon
})F(T_{l+1}(u_{\varepsilon })).D(h(u)\varphi ),
\end{equation}
\noindent and conclude that
\begin{equation}\label{eq5.30}
\lim\limits_{\varepsilon \rightarrow
0}I_{\varepsilon ,l}^{4,2}=\int_{\Omega }h_{l}(u)F(u).D(h(u)\varphi ).
\end{equation} 

\underline{Step 2:} Passage to the limit with $l\rightarrow \infty $.\newline
Combining $(\ref{eq5.21})$ and $(\ref{eq5.222})-(\ref{eq5.30})$ we deduce that
\begin{equation}\label{eq5.31}
I_{l}^{1}+I_{l}^{2}+I_{l}^{3}+I_{l}^{4}+I_{l}^{5}=I_{l}^{6}
\end{equation}
where
\begin{equation*}
\begin{array}{lll}
I_{l}^{1}=\int_{\Omega }bh_{l}(u)h(u)\varphi , &  & I_{l}^{2}=\int_{\Omega
}h_{l}(u)a(x,DT_{l+1}(u)).D(h(u)\varphi ), \\
&  &  \\
|I_{l}^{3}|\leq C_{2}|l^{-(1-1/\bar{p})}||h||_{\infty }||\varphi ||_{\infty
}, &  & I_{l}^{4}=\int_{\Omega }h_{l}(u)F(u).D(h(u)\varphi ), \\
&  &  \\
I_{l}^{5}=\int_{\Omega }h_{l}^{\prime }(u)F(u).Duh(u)\varphi , &  &
I_{l}^{6}=\int_{\Omega }fh_{l}(u)h(u)\varphi .
\end{array}
\end{equation*}
Obviously, we have
\begin{equation}\label{eq5.32}
\lim_{\varepsilon \rightarrow \infty }I_{l}^{3}=0.
\end{equation}
Choosing $m>0$ such that $supp$ $h\subset \lbrack -m,m],$ we can replace $u$
by $T_{m}(u)$ in $I_{l}^{1},I_{l}^{2},...,I_{l}^{6},$ and
\begin{equation*}
h_{l}^{\prime }(u)=h_{l}^{\prime
}(T_{m}(u))=0\hspace*{2mm}if\hspace*{2mm}l+1>m,\text{ }%
h_{l}(u)=h_{l}(T_{m}(u))=0\hspace*{2mm}if\hspace*{2mm}l>m.
\end{equation*}
Therefore, letting $l\rightarrow \infty $ and combining $(\ref{eq5.31})$ with $(\ref{eq5.32})
$ we obtain

\begin{equation}\label{eq5.33}
\int_{\Omega }bh(u)\varphi +\int_{\Omega }(a(x,Du)+F(u)).D(h(u)\varphi
)=\int_{\Omega }fh(u)\varphi
\end{equation}
for all $h\in C_{c}^{1}(\mathbb{R})$ and all $\varphi \in W_{0}^{1,{%
\overrightarrow{p}}}(\Omega )\cap L^{\infty }(\Omega )$. \newline
\underline{Step 3:} Subdifferential argument\newline

It is left to prove that $u(x)\in D(\beta (x))$ and $b(x)\in \beta (u(x))$
for almost all $x\in \Omega $. Since $\beta $ is a maximal monotone graph,
there exist a convex, l.s.c and proper function $j:\mathbb{R}\rightarrow
\lbrack 0,\infty ],$ such that
\begin{equation*}
\beta (r)=\partial j(r)\hspace*{3mm}\text{for all\hspace*{2mm}}r\in \mathbb{R}.
\end{equation*}
According to \cite{ref:8}, for $0<\varepsilon \leq 1,j_{\varepsilon }:%
\mathbb{R}\rightarrow \mathbb{R}$ defined by $j_{\varepsilon
}(r)=\int_{0}^{r}\beta _{\varepsilon }(s)ds$ has the following properties as
in \cite{ref:1}

$i)$ For any $0<\varepsilon \leq 1,j_{\varepsilon }$ is convex and
differentiable for all $r\in \mathbb{R}$, such that
\begin{equation*}
j_{\varepsilon }^{\prime }(r)=\beta _{\varepsilon }(r)\text{ for all }r\in
\mathbb{R}\text{ and any }0<\varepsilon \leq 1.
\end{equation*}
$ii)$ $j_{\varepsilon }(r)\rightarrow j(r)$ for all $r\in \mathbb{R}$ as $%
\varepsilon \rightarrow 0.$\newline
\noindent From $i)$, it follows that for any $0<\varepsilon \leq 1$
\begin{equation}\label{eq5.34}
j_{\varepsilon }(r)\geq j_{\varepsilon }(T_{1/\varepsilon }(u_{\varepsilon
}))+(r-T_{1/\varepsilon }(u_{\varepsilon }))\beta _{\varepsilon
}(T_{1/\varepsilon }(u_{\varepsilon }))
\end{equation}
holds for all $r\in \mathbb{R}$ and almost everywhere in $\Omega $. \newline
Let $E\cup \Omega $ be an arbitrary measurable set and $\chi _{E}$ its
characteristic function. We fix $\varepsilon _{0}>0.$ Multiplying $(\ref{eq5.34})$
by $h_{l}(u_{\varepsilon })\chi _{E}$, integrating over $\Omega $ and using $%
ii)$, we obtain
\begin{equation}\label{eq5.35}
j(r)\int_{E}h_{l}(u_{\varepsilon })\geq \int_{E}j_{\varepsilon
_{0}}(T_{l+1}(u_{\varepsilon }))h_{l}(u_{\varepsilon
})+(r-T_{l+1}h_{l}(u_{\varepsilon })\beta _{\varepsilon }(T_{1/\varepsilon
}(u_{\varepsilon }))
\end{equation}
for all $r\in \mathbb{R}$ and all $0<\varepsilon <min(\varepsilon _{0},%
\dfrac{1}{l}).$ \newline
As $\varepsilon \rightarrow 0$, taking into account that $E$ arbitrary we
obtain from $(\ref{eq5.35})$

\begin{equation}\label{eq5.36}
j(r)h_{l}(u)\geq j_{\varepsilon
_{0}}(T_{l+1}(u))h_{l}(u)+bh_{l}(u)(r-T_{l+1}(u))
\end{equation}
for all $r\in \mathbb{R}$ and almost everywhere in $\Omega $. \newline
Passing to the limit with $l\rightarrow \infty $ and then with $\varepsilon
_{0}\rightarrow 0$ in $(\ref{eq5.36})$ finally yields
\begin{equation}\label{eq5.37}
j(r)\geq j(u(x))+b(x)(r-u(x))
\end{equation}
for all $r\in \mathbb{R}$ and almost everywhere in $\Omega $, hence $u\in
D(\beta )$ and $b\in \beta (u)$ for almost everywhere in $\Omega $. With
this last step the proof of Theorem \ref{5.1} is concluded.
\end{proof}

\section{Case where $f\in L^{1}(\Omega )$}

\subsection{Approximate solution for $L^1$- data}

The comparison principle from proposition will be the tool in second
approximation procedure. For $f\in L^{1}(\Omega )$ and $m,n\in \mathbb{N}$
let $f_{m,n}\in L^{\infty }(\Omega )$ be defined as in Section 3. Using
Propsition \ref{4.3}, we deduce that for any $m,n\in \mathbb{N,}$ there
exists $u_{m,n}\in W_{0}^{1,{\overrightarrow{p}}}(\Omega ),b_{m,n}\in
L^{\infty }(\Omega )$, such that $(u_{m,n},b_{m,n})$ is a renormalized
solution of $(E,f_{m,n})$. Therefore
\begin{equation}\label{eq6.1}
\int_{\Omega }b_{m,n}h(u_{m,n})\phi +\int_{\Omega
}(a(x,Du_{m,n})+F(u_{m,n})).D(h(u_{m,n})\phi )=f_{m,n}h(u_{m,n})\phi
\end{equation}
holds for all $m,$ $n\in \mathbb{N},$ $h\in C_{c}^{1}(\mathbb{R}),\phi \in
W_{0}^{1,{\overrightarrow{p}}}(\Omega )\cap L^{\infty }(\Omega )$. In the
next lemma, we give a priori estimates that will be important in the the
following:

\begin{lemma}\label{6.1}
\label{ref:1111} For $m,n\in \mathbb{N,}$ let $(u_{m,n},b_{m,n})$ be a
renormalized solution of $(E,f_{m,n})$ . Then,\newline
$i)$ For any $k>0$ we have,
\begin{equation}\label{eq6.2}
\sum_{i=1}^{N}\int_{\Omega }|DT_{k}(u_{m,n})|^{p_{i}}\leq \dfrac{k}{\gamma }%
\Vert f\Vert _{1}
\end{equation}
$ii)$ for any $k>0$, there exists a constant $C_{3}(k)>0$, not depending on $%
m,n\in \mathbb{N}$, such that
\begin{equation}\label{eq6.3}
\sum_{i=1}^{N}\int_{\Omega }|DT_{k}(u_{m,n})|^{p_{i}}\leq C_{3}(k).
\end{equation}
\noindent $iii)$ For $m,n\in \mathbb{N}$, we have: 
\begin{equation}\label{eq6.333}
\Vert b_{m,n}\Vert_1\leq \Vert f \Vert_1 .
\end{equation}

\end{lemma}

\begin{proof}
For $l,k>0$, we plug $h_{l}(u_{m,n})T_{k}(u_{m,n})$ as a test function in $
(\ref{eq6.1})$. Then $i)$ and $ii)$ follows with similar arguments as used in the
proof of Lemma \ref{5.4}. To prove $iii)$, we neglet the positve term
\begin{equation*}
\int_{\Omega }a(x,DT_{k}(u_{m,n}))DT_{k}(u_{m,n})
\end{equation*}
and keep
\begin{equation}\label{eq6.4}
\int_{\Omega }b_{m,n}T_{k}(u_{m,n})\leq \int_{\Omega }f_{m,n}T(u_{m,n}).
\end{equation}
Since $b_{m,n}\in \beta (u_{m,n})$ a.e. in $\Omega $, it follows from $(\ref{eq6.4})$
that
\begin{equation}\label{eq6.5}
\int_{|u_{m,n}|>k}|b_{m,n}|\leq \int_{\Omega }|f|.
\end{equation}

and we find $iii)$ by passing to the limit with $k\rightarrow 0$.
\end{proof}

By definition we have
\begin{equation}\label{eq6.6}
f_{m,n}\leq f_{m+1,n}\hspace{4mm}and\hspace{4mm}f_{m,n+1}\leq f_{m,n}
\end{equation}
From Propostion \ref{5.2} it follows that
\begin{equation}\label{eq6.7}
u_{m,n}^{\varepsilon }\leq u_{m+1,n}^{\varepsilon }\hspace{4mm}and\hspace{4mm%
}u_{m,n+1}^{\varepsilon }\leq u_{m,n}^{\varepsilon },
\end{equation}
almost everywhere in $\Omega $ for any $m,n\in \mathbb{N}$ and all $%
\varepsilon >0$. \newline
Hence passing to the limit with $\varepsilon \rightarrow 0$ in $(\ref{eq6.7})$ yields
\begin{equation}\label{eq6.8}
u_{m,n}\leq u_{m+1,n}\hspace{4mm}and\hspace{4mm}u_{m,n+1}\leq u_{m,n},
\end{equation}
almost everywhere in $\Omega $ for any $m,n\in \mathbb{N}$. \newline
Setting $b_{\varepsilon }:=\beta _{\varepsilon }(T_{1/\varepsilon }
(u_{\varepsilon })),$ using $(\ref{eq6.7})$, Remark \ref{5.3} and the fact that $%
b_{m,n}^{\varepsilon }\rightharpoonup b_{m,n}$ in $L^{\infty }(\Omega )$ and
since this convergence preserves order we get
\begin{equation}\label{eq6.9}
b_{m,n}\leq b_{m+1,n}\hspace{4mm}and\hspace{4mm}b_{m,n+1}\leq b_{m,n}
\end{equation}
almost everywhere in $\Omega $ for any $m,n\in \mathbb{N}$. By $(\ref{eq6.9})$ and $%
(\ref{eq6.333})$, for any $n\in \mathbb{N}$ there exist $b^{n}\in L^{1}(\Omega )$ such
that $b_{m,n}\rightarrow b^{n}$ and $m\rightarrow \infty $ in $L^{1}(\Omega )
$ and almost everywhere and $b\in L^{1}(\Omega )$, such that $%
b^{n}\rightarrow b$ as $n\rightarrow \infty $ in $L^{1}(\Omega )$ and almost
every where in $\Omega $. By $(\ref{eq6.8})$, the sequence $(u_{m,n})_{m}$ is
monotone increasing, hence, for any $n\in \mathbb{N},u_{m,n}\rightarrow u^{n}
$ almost everywhere in $\Omega $, where $u^{n}:\Omega \rightarrow \mathbb{%
\overline{R}}$ is a mesurable function. In order to show that $u$ is finite
almost everywhere we will give an estimate on the level sets of $u_{m,n}$ in
the next lemma:

\begin{lemma}\label{6.2}
For $m,n\in \mathbb{N}$ let $(u_{m,n},b_{m,n})$ be a renormalized solution
of $(E,f_{m,n})$. Then, there exists a constant $C_{4}>0$, not depending on $%
m,n\in \mathbb{N}$, such that
\begin{equation}\label{eq6.10}
|\{|u_{m,n}|\geq l\}|\leq C_{4}l^{\frac{1}{\overline{p}}-1}
\end{equation}
for all $l\geq 0$.
\end{lemma}

\begin{proof}
With the same arguments as in remark $\ref{5.5}$ we obtain

\begin{equation}\label{eq6.11}
|\{|u_{m,n}|\}|\leq C(\overline{p},N)l^{\overline{p}^{-}}(\sum_{i=1}^{N}%
\int_{\Omega }\left| DT_{k}(u_{m,n})\rvert ^{p_{i}}+|\Omega |\right)
\end{equation}
for all $m,n\in \mathbb{N}$ where $C(\overline{p},N)$ is the constant from
Sobolev embedding in $(\ref{eq2.5})$. Now we plug $(\ref{eq6.2})$ into $(\ref{eq6.11})$ to obtain $%
(\ref{eq6.10})$. Note that, as $(u_{m,n})_m$ is pointwise increasing with respect to $m$,
\begin{equation}\label{eq6.12}
\lim_{m\rightarrow \infty }|\{u_{m,n}\geq l\}|=|\{u^{n}\geq l\}|
\end{equation}
and
\begin{equation}\label{eq6.13}
\lim_{m\rightarrow \infty }|\{u_{m,n}\leq -l\}|=|\{u^{n}\leq -l\}|.
\end{equation}
Combining $(\ref{eq6.10})$ with $(\ref{eq6.12})$ and $(\ref{eq6.13})$ we get
\begin{equation}\label{eq6.14}
|\{u^{n}\leq -l\}|+|\{u^{n}>l\}|\leq C_{4}l^{\frac{1}{\overline{p}}-1}
\end{equation}
for any $l\geq 1$, hence $u^{n}$ is finite almost everywhere for $n\in
\mathbb{N}$. By the same arguments we get
\begin{equation}\label{eq6.15}
|\{u<-l\}|+|\{u>l\}|\leq C_{4}l^{\frac{1}{\overline{p}}-1}
\end{equation}
from $(\ref{eq6.14})$, hence $u$ is finite almost eveyrywhere. Now, since $%
b_{m,n}\in \beta (u_{m,n})$ almost everywhere in $\Omega $ it follows by a
subdifferential argument that $b^{n}\in \beta (u^{n})$ and $b\in \beta (u)$
a,e. in $\Omega $.

\end{proof}

\begin{remark}\label{6.3}
If $(u_{m,n},b_{m,n})$ is renormalized solution of $(E,f_{m,n})$, using
\newline
$h_{\nu }(u_{m,n})T_{k}(u_{m,n}-T_{l}(u_{m,n}))$ as a test function in $(\ref{eq6.1})
$ , neglecting positive terms and passing to the limit with $\nu \rightarrow
\infty $ we obtain
\begin{equation}\label{eq6.16}
\int_{\{l<|u_{m,n}|<l+k\}}a(x,Du_{m,n}).Du_{m,n}\leq k\Bigg\lgroup%
\int_{\{|u_{m,n}|>l\}\cap \{|f|<\sigma \}}|f|+\int_{\{|f|>\sigma \}}|f|%
\Bigg\rgroup
\end{equation}
for any $k,\sigma >0,l$. Now applying $(\ref{eq6.10})$ to $(\ref{eq6.16})$, we find that
\begin{equation}\label{eq6.17}
\int_{\{l<|u_{m,n}|<l+k\}}a(x,Du_{m,n}).Du_{m,n}\leq \sigma kC_{4}l^{\frac{1%
}{\overline{p}}-1}+k\int_{\{|f|>\sigma \}}|f|
\end{equation}
holds for any $k,\sigma >0,l\geq 0$ uniformly in $m,n\in \mathbb{N}.$
\end{remark}

\subsection{Basic convergence results}

\begin{lemma}\label{6.4}
For $m,n\in \mathbb{N}$ let $(u_{m,n},b_{m,n})$ be a renormalized soltuion
of $(E,F_{m,n})$. There exists a subsequence $(m(n))_{n}$ such that setting $%
f_{n}:=f_{m(n),n},b_{n}:=b_{m(n),n},$ $u_{n}:=u_{m(n),n}$  we have
\begin{equation}\label{eq6.18}
u_{n}\rightarrow u\hspace{4mm}almost\hspace{2mm}everywhere\hspace{2mm}in%
\hspace{2mm}\Omega .
\end{equation}
Moreover, for any $k>0,$
\begin{equation}\label{eq6.19}
T_{k}(u_{n})\rightarrow T_{k}(u)\hspace{2mm}in\hspace{2mm}W_{0}^{1,{%
\overrightarrow{p}}}(\Omega ),
\end{equation}
\begin{equation}\label{eq6.20}
DT_{k}(u_{n})\rightharpoonup DT_{k}(u)\hspace{2mm}in\hspace*{2mm}%
\prod_{i=1}^{N}L^{p_{i}}(\Omega ),
\end{equation}
\begin{equation}\label{eq6.21}
a(x,DT_{k}(u_{n}))\rightharpoonup a(x,DT_{k}(u))\hspace{2mm}%
in\hspace*{2mm}\prod_{i=1}^{N}L^{p_{i}^{\prime }}(\Omega ),
\end{equation}
as $n\rightarrow \infty $.
\end{lemma}

\begin{proof}
We construct a subsequence $(m(n))_{n}$, such that
\begin{equation*}
\arctan (u_{m(n),n})\rightarrow \arctan (u),
\end{equation*}
\begin{equation*}
b_{n}:=b_{m(n),n}\rightarrow b,
\end{equation*}
\begin{equation*}
f_{n}:=f_{m(n),n}\rightarrow f
\end{equation*}
as $n\rightarrow \infty $ in $L^{1}(\Omega )$ and almost everywhere in $%
\Omega $. It follows that $(\ref{eq6.18})$ and $(\ref{eq6.19})$ hold. Combining $(\ref{eq6.19})$
with $(\ref{eq6.3})$ we get $T_{k}(u)\in W_{0}^{1,{\overrightarrow{p}}}(\Omega ), T_k(u_n) \rightarrow T_k(u) \in W_{0}^{1,{\overrightarrow{p}}}(\Omega ) $
and $(\ref{eq6.20})$ holds for any $k>0$. From $(\ref{eq6.2})$ and $(H_{2}),$ it follows
that for fixed $k>0$, given any subsequence of $(a(x,DT_{k}(u_{n})))_{n}$ there exists a subsequence, still denoted by 
such that $a(x,DT_{k}(u_{n}))_{n}$, such that  
\begin{equation*}
a(x,DT_{k}(u_{n}))_{n}\rightharpoonup \Phi _{k}\hspace{3mm}in\hspace{2mm}%
\prod_{i=1}^{N}L^{p_{i}^{\prime }}(\Omega )
\end{equation*}
as $n\rightarrow \infty $. Since $h_{l}(u_{n})(T_{k}(u_{n})-T_{k}(u))$ is an
admissible test function in $(\ref{eq6.1})$,
\begin{equation}\label{eq6.22}
\lim_{n\rightarrow \infty }\sup \int_{\Omega
}a(x,DT_{k}(u_{n}))D(T_{k}(u_{n})-T_{k}(u))\leq 0.
\end{equation}
Then, $(\ref{eq6.21})$ follows with the same arguments as int the proof of Lemma \ref{5.6}.
\end{proof}

\begin{remark}
With the same arguments as in Remark \ref{5.7}, we have
\begin{equation}\label{eq6.23}
\lim_{n\rightarrow \infty }\int_{\Omega
}a(x,DT_{k}(u_{n}) - a(x,DT_{k}(u))). D(T_{k}(u_{n})-T_{k}(u))=0,
\end{equation}
\begin{equation}\label{eq6.24}
\lim_{l\rightarrow \infty }\int_{\{l<|u|<l+1\}}a(x,Du).
Du=0.
\end{equation}
\end{remark}

\subsection{Conclusion of the proof of Theorem \ref{4.1}}

It is left to prove thet $(u,b)$ satisfies
\begin{equation}\label{eq6.25}
\int_\Omega bh(u)\phi + \int_{\Omega} (a(x,Du)+F(u)). D(h(u)\phi)=
\int_{\Omega}fh(u)\phi.
\end{equation}
for all $h\in C^1_c(\mathbb{R})$ and $\phi \in W^{1,{\overrightarrow{p}}%
}_0(\Omega) \cap L^\infty (\Omega )$. To this end, we take $h\in C^1_c(%
\mathbb{R})$ and $\phi \in W^{1,{\overrightarrow{p}}}_0(\Omega) \cap
L^\infty (\Omega )$ arbitrary and plug $h_l(u_n)h(u)\phi$ into $(\ref{eq6.1})$ to
obtain

\begin{equation}\label{eq6.26}
I_{n,l}^{1}+I_{n,l}^{2}+I_{n,l}^{3}=I_{n,l}^{4},
\end{equation}
where
\begin{equation*}
I_{n,l}^{1}=\int_{\Omega }b_{n}h_{l}(u_{n})h(u)\phi ,
\end{equation*}

\begin{equation*}
I^{2}_{n,l}= \int_{\Omega} a(x,Du_n).D(h_l(u_n)h(u)\phi),
\end{equation*}

\begin{equation*}
I^{3}_{n,l}= \int_{\Omega} F(u_n).D(h_l(u_n)h(u)\phi),
\end{equation*}

\begin{equation*}
I^{4}_{\varepsilon,l}= \int_{\Omega} f_nh_l(u_n)h(u)\phi.
\end{equation*}

\textbf{Step 1}. Passing to the limit as $n\rightarrow \infty $, applying
the convergence results from Lemma $\ref{6.4}$ we get
\begin{equation}\label{eq6.27}
\lim_{n\rightarrow \infty }I_{n,l}^{1}=\int_{\Omega }bh_{l}(u)h(u)\phi
,\quad \lim_{n\rightarrow \infty }I_{n,l}^{4}=\int_{\Omega
}fh_{l}(u)h(u)\phi .
\end{equation}
Let us write
\begin{equation}\label{eq6.28}
I_{n,l}^{2}=I_{n,l}^{2,1}+I_{n,l}^{2,2},
\end{equation}
where
\begin{equation}\label{eq6.29}
I_{n,l}^{2,1}=\int_{\Omega }h_{l}(u_{n})a(x,Du_{n}).D(h(u)\phi ),\quad
I_{n,l}^{2,2}=\int_{\Omega }h_{l}^{\prime }(u_{n})a(x,Du_{n}).Du_{n}h(u)\phi
.
\end{equation}
With similar arguments as in the proof of $(\ref{eq5.25})$ it follows that
\begin{equation}\label{eq6.30}
\lim_{n\rightarrow \infty }I_{n,l}^{2,1}=\int_{\Omega
}h_{l}(u)a(x,Du).D(h(u)\phi ).
\end{equation}
By $(\ref{eq6.17})$, we get the estimate
\begin{equation}\label{eq6.31}
|\lim_{n\rightarrow \infty }I_{n,l}^{2,2}|\leq \Vert h\Vert _{\infty }\Vert
\phi \Vert _{\infty }\big(\delta C_{4}l^{\frac{1}{\overline{p}}%
-1}+\int_{\{|f|>\delta \}}|f|\big),
\end{equation}
for all $n\int \mathbb{N}$ and all $l\geq 1$, $\delta >0$. Next, we write
\begin{equation*}
I_{n,l}^{3}=I_{n,l}^{3,1}+I_{n,l}^{3,2},
\end{equation*}
where
\begin{equation}\label{eq6.32}
\lim_{n\rightarrow \infty }I_{n,l}^{3,1}=\int_{\Omega
}h_{l}(u)F(u).D(h(u)\phi ),\lim_{n\rightarrow \infty
}I_{n,l}^{3,2}=\int_{\Omega }h_{l}^{\prime }(u)F(u).Duh(u)\phi ,
\end{equation}
follows with the same arguments as in $(\ref{eq5.26})-(\ref{eq5.30})$.\newline
\textbf{Step 2}. Passing to the limit as $l\rightarrow \infty $. Combining
(\ref{eq6.26}) with (\ref{eq6.27})-(\ref{eq6.32}) we get for all $\delta >0$ and all $l\geq 1$
\begin{equation}\label{eq6.33}
I_{l}^{1}+I_{l}^{2}+I_{l}^{3}+I_{l}^{4}+I_{l}^{5}=I_{l}^{6},
\end{equation}
where
\begin{equation*}
I_{l}^{1}=\int_{\Omega }bh_{l}(u)h(u)\phi ,\quad I_{l}^{2}=\int_{\Omega
}h_{l}(u)a(x,DT_{l+1}(u)).D(h(u)\phi )
\end{equation*}
\begin{equation*}
|I_{l}^{3}|\leq \Vert h\Vert _{\infty }\Vert \phi \Vert _{\infty }\big(%
\delta C_{4}l^{\frac{1}{\overline{p}}-1}+\int_{\{|f|>\delta \}}|f|\big),
\end{equation*}
for any $\delta >0$ and
\begin{equation*}
I_{l}^{4}=\int_{\Omega }h_{l}^{\prime }(u)F(u)h(u)\phi Du,\text{ }%
I_{l}^{5}=\int_{\Omega }h_{l}(u)F(u).D(h(u)\phi ),\text{ }%
|I_{l}^{6}|=\int_{\Omega }fh_{l}(u)h(u)\phi .
\end{equation*}
Choosing $m>0$ such that $supph\subset  [-m,m]$, we can
replace $u$ by $T_{m}(u)$ in $I_{l}^{1},I_{l}^{2},\dots ,I_{l}^{6}$ hence
\begin{equation}\label{eq6.34}
\lim_{l\rightarrow \infty }I_{l}^{1}=\int_{\Omega }bh(u)\phi
,\lim_{l\rightarrow \infty }I_{l}^{2}=\int_{\Omega }a(x,Du).D(h(u)\phi ),
\end{equation}
\begin{equation}\label{eq6.35}
\lim_{l\rightarrow \infty }|I_{l}^{3}|\leq \Vert h\Vert _{\infty }\Vert \phi
\Vert _{\infty }\int_{\{|f|>\sigma \}}|f|,\lim_{l\rightarrow \infty
}I_{l}^{4}=0,
\end{equation}
\begin{equation}\label{eq6.36}
\lim_{l\rightarrow \infty }|I_{l}^{5}|=\int_{\Omega }F(u).D(h(u)\phi
),\lim_{l\rightarrow \infty }|I_{l}^{6}|=\int_{\Omega }fh(u)\phi ,
\end{equation}
for all $\delta >0$. Combining (\ref{eq6.33}) with (\ref{eq6.34})-(\ref{eq6.36}) we finally deduce
that (6.1) holds for all $h\in C_{C}^{1}(\mathbb{R})$ and all $\phi \in
W_{0}^{1,{\overrightarrow{p}}}(\Omega )\cap L^{\infty }(\Omega )$.\newline
Hence $(u,b)$ satisfies (R1), (R2) and (R3) and the proof of the theorem is
completed.

\subsection{Proof of Theorem \ref{4.2} (Uniqueness)}

\begin{lemma}
For $f,\widetilde{f}\in L^{1}(\Omega )$ let $(u,b)$, $(\widetilde{u},%
\widetilde{b})$ be the renormalized solutions to $(E,f)$ and $(E,\widetilde{f%
})$ respectively, then
\begin{equation}\label{eq6.37}
\int_{\Omega }(b-\widetilde{b})Sign_{0}^{+}(u-\widetilde{u})dx\leq
\int_{\Omega }(f-\widetilde{f})Sign_{0}^{+}(u-\widetilde{u})dx,
\end{equation}
\end{lemma}

\textit{Proof}. For $\delta >0$ let $H_{\delta }^{+}$ be a Lipschitz
approximation of the $sign_{0}^{+}$ function. Since $(u,b)$, $(\widetilde{u},%
\widetilde{b})$ are renormalized solutions, it follows that 
\begin{center}
$T_{l+1}(u),T_{l+1}(\widetilde{u})\in W_{0}^{1,{\overrightarrow{p}}}(\Omega
)\cap L^{\infty }(\Omega )$ for all $l>0$.
\end{center}

Hence $H_{\delta }^{+}(T_{l+1}(u)-T_{l+1}(\widetilde{u}))\in W_{0}^{1,{%
\overrightarrow{p}}}(\Omega )\cap L^{\infty }(\Omega )$ for $l,\delta >0$.
\newline
Now, we choose $H_{\delta }^{+}(T_{l+1}(u)-T_{l+1}(\widetilde{u}))$ as a
test function in the renormalized formulation with $h=h_{l}$ for $(u,b)$ and
for $(\widetilde{u},\widetilde{b})$ respectively. Subtracting the resulting
equalities, we obtain
\begin{equation}\label{eq6.38}
I_{l,\delta }^{1}+I_{l,\delta }^{2}+I_{l,\delta }^{3}+I_{l,\delta
}^{4}+I_{l,\delta }^{5}=I_{l,\delta }^{6},
\end{equation}
where $K=\{0<T_{l+1}(u)-T_{l+1}<\delta \}$ and
\begin{equation*}
I_{l,\delta }^{1}=\int_{\Omega }(bh_{l}(u)-\widetilde{b}h_{l}(\widetilde{u}%
))H_{\delta }^{+}(T_{l+1}(u)-T_{l+1}(\widetilde{u}))dx,
\end{equation*}
\begin{equation*}
I_{l,\delta }^{2}=\int_{\Omega }(h_{l}^{\prime }(u)a(x,Du).Du-h_{l}^{\prime
}(\widetilde{u})a(x,D\widetilde{u}).D\widetilde{u}).H_{\delta
}^{+}(T_{l+1}(u)-T_{l+1}(\widetilde{u}))dx,
\end{equation*}
\begin{equation*}
I_{l,\delta }^{3}=\frac{1}{\delta }\int_{K}(h_{l}(u)a(x,Du)-h_{l}(\widetilde{%
u})a(x,D\widetilde{u})).D(T_{l+1}(u)-T_{l+1}(\widetilde{u}))dx,
\end{equation*}
\begin{equation*}
I_{l,\delta }^{4}=\int_{\Omega }(h_{l}^{\prime }(u)F(u).Du-h_{l}^{\prime }(%
\widetilde{u})F(\widetilde{u}).D\widetilde{u})H_{\delta
}^{+}(T_{l+1}(u)-T_{l+1}(\widetilde{u}))dx,
\end{equation*}
\begin{equation*}
I_{l,\delta }^{5}=\frac{1}{\delta }\int_{K}(h_{l}(u)F(u)-h_{l}(\widetilde{u}%
)F(\widetilde{u})).D(T_{l+1}(u)-T_{l+1}(\widetilde{u}))dx,
\end{equation*}
\begin{equation*}
I_{l,\delta }^{6}=\int_{\Omega }(fh_{l}(u)-\widetilde{f}h_{l}(\widetilde{u}%
))H_{\delta }^{+}(T_{l+1}(u)-T_{l+1}(\widetilde{u}))dx.
\end{equation*}
Using the same arguments as in \cite{ref:1} i.e., neglecting the nonnegative
part of  $I_{l,\delta }^{3}$ and using that $F$ is locally Lipschitz
continuous, we can pass to the limit as $\delta \rightarrow 0$. \newline
Using the energy dissipation condition $(R_3)$ we can pass the limit as $%
l\rightarrow \infty $ and obtain (\ref{eq6.37}).

Now we are in position to give the proof of Theorem \ref{4.2}:\newline
Assuming $f=\widetilde{f}$, from lemma $6.6$ we get
\begin{equation}\label{eq6.39}
\int_{\Omega }(b-\widetilde{b})sign_{0}^{+}(u-\widetilde{u})dx\leq 0,
\end{equation}
hence $(b-\widetilde{b})sign_{0}^{+}(u-\widetilde{u})=0$ almost everywhere
in $\Omega $. Now, let us write \newline
$\Omega =\Omega _{1}\cup \Omega _{2}$, where $\Omega _{1}:=\{x\in \Omega
:sign_{0}^{+}(u(x)-\widetilde{u}(x))=0\},\Omega _{2}:=\{x\in \Omega :(b(x)-%
\widetilde{b}(x))=0\}$. Since $r\mapsto \beta (x,r)$ is strictly increasing
for a,e. $x\in \Omega $, we can define the function $\beta ^{-1}:\mathbb{%
R}\rightarrow \rightarrow \mathbb{R}$ such that $\beta ^{-1}=s$ for all $%
(r,s)\in \mathbb{R}^{2}$ such that $r\in \beta (x,r)$ for a,e. $x\in \Omega $%
. For a,e. $x\in \Omega _{2}$ we have $b(x)=\widetilde{b}(x)$, hence $%
u(x)=\beta ^{-1}(b(x))=\beta ^{-1}(\widetilde{b}(x))=\widetilde{u}(x)
$. Therefore, $u(x)=\widetilde{u}(x)$
a,e. in $\Omega _{2}$  and $sign_{0}^{+}(u-\widetilde{u})=0$. Interchanging the roles of $u$ and $\widetilde{u}$
and respeating the arguments, we get $sign_{0}^{+}(\widetilde{u}-u)=0$ a,e.
in $\Omega $ and we finally arrive at $u=\widetilde{u}$ a,e. in $\Omega $.
Now, we write the renormalized formulation for $(u,b)$ and $(\widetilde{u},%
\widetilde{b})$ respectively. Substracting the resulting equalities, we
obtain
\begin{equation*}
\int_{\Omega }(b-\widetilde{b})h(u)\varphi dx=0
\end{equation*}
for all $h\in C_{c}^{1}(\mathbb{R})$ and all $\varphi \in C_{c}^{\infty
}(\Omega )$. Choosing $h(u)=h_{l}(u)$ and passing to the limit with $%
l\rightarrow \infty $ we find $b=\widetilde{b}$ a,e. in $\Omega $.

\section{Proof of Proposition \ref{4.3}}

Note that for $\varepsilon ,k>0$, $h_{l}(u)\dfrac{1}{\varepsilon }%
T_{\varepsilon }(u-T_{k}(u))$ as a test function in $(\ref{eq3.2})$. Neglecting
positive terms and passing to the limit with $l\rightarrow \infty $, we
obtain
\begin{equation}\label{eq7.1}
\dfrac{1}{\varepsilon }\sum_{i=1}^{N}\int_{k<|u|<k+\varepsilon
}|Du|^{p_{i}}\leq \lVert f\lVert _{N}(\phi (k))^{(N-1)/N},
\end{equation}
where $\phi (k):=|\{|u|>k\}|$ for $k>0$. Now we use similar arguments as in \cite{ref:1}. We apply the continuous embedding of $W_{0}^{1,1}(\Omega )$
into $L^{N/N-1}(\Omega )$ and the H\"{o}lder inequality to get
\begin{equation}\label{eq7.2}
\dfrac{1}{\varepsilon C_{N}}\lVert T_{\varepsilon }(u-T_{k}(u))\lVert _{%
\frac{N}{N-1}}\leq \Bigg\lgroup\dfrac{\phi (k)-\phi (k+\varepsilon )}{%
\varepsilon }\Bigg\rgroup^{1/(p^{-})^{\prime }}\Bigg \lgroup\dfrac{1}{%
\varepsilon }\int_{k<|u|<k+\varepsilon }|Du|^{p^{-}}\Bigg\rgroup^{1/p^{-}},
\end{equation}
where $C_{N}>0$ is the constant coming from the Sobolev embedding. \newline
Notice that
\begin{equation}\label{eq7.3}
\dfrac{1}{\varepsilon }\sum_{i=1}^{N}\int_{k<|u|<k+\varepsilon
}|Du|^{p^{-}}\leq \dfrac{\phi (k)-\phi (k+\varepsilon )}{\varepsilon }+%
\dfrac{1}{\varepsilon }\sum_{i=1}^{N}\int_{k<|u|<k+\varepsilon }|Du|^{p_{i}},
\end{equation}
hence, from $(\ref{eq7.1}),(\ref{eq7.2})$ and $(\ref{eq7.3})$ we deduce that\newline
\begin{equation}\label{eq7.4}
\dfrac{1}{\varepsilon C_{N}}\lVert T_{\varepsilon }(u-T_{k}(u))\lVert _{%
\frac{N}{N-1}}\leq \Bigg\lgroup\dfrac{\phi (k)-\phi (k+\varepsilon )}{%
\varepsilon }\Bigg \rgroup^{1/(p^{-})^{\prime }}\Bigg\lgroup\dfrac{\phi
(k)-\phi (k+\varepsilon )}{\varepsilon }+\lVert f\lVert _{N}(\phi
(k))^{(N-1)/N}\Bigg\rgroup^{1/(p^{-})}.
\end{equation}

From $(\ref{eq7.4})$ and Young's inequality with $\alpha >0$ it follows that
\begin{equation}\label{eq7.5}
\dfrac{1}{C_{N}C}(\phi (k+\varepsilon ))^{(N-1)/N}-\dfrac{\alpha ^{p^{-}}}{%
p^{-}C}\lVert f\lVert _{N}(\phi (k))^{(N-1)/N}-\dfrac{\phi (k)-\phi
(k+\varepsilon )}{\varepsilon }\leq 0,
\end{equation}
where
\begin{equation*}
C:=\Bigg(\dfrac{1}{\alpha ^{(p^{-})^{\prime }}(p^{-})^{\prime }}+\dfrac{%
\alpha ^{p^{-}}}{p^{-}}\Bigg )>0.
\end{equation*}
The mapping $(0,\infty )\ni k\rightarrow \phi (k)$ is non-increasing and
therefore of bounded variation, hence it is differentiable almost everywhere
on $(0,\infty )$ with $\phi ^{\prime }\in L_{loc}^{1}(0,\infty )$. Since it
is also continuous from the right, we can pass to the limit with $%
\varepsilon \downarrow 0$ in $(\ref{eq7.5})$ to find
\begin{equation}\label{eq7.6}
C^{\prime \prime }(\phi (k))^{(N-1)/N}+\phi ^{\prime }(k)\leq 0
\end{equation}
for almost every $k>0$ and $\alpha >0$ choosen small enough such that
\begin{equation*}
C^{\prime \prime }:=\Bigg(\dfrac{C_{N}}{C}-\dfrac{\alpha ^{p^{-}}}{p^{-}C}%
\lVert f\lVert _{N}\Bigg)>0.
\end{equation*}
Now, the conclusion of the proof follows by contradiction. We assume that $%
\phi (k)>0$ for each $k>0$. For $k>0$ fixed, we choose $k_{0}<k.$ From $(\ref{eq7.6})
$ it follows that
\begin{equation}\label{eq7.7}
\dfrac{1}{N}C^{\prime \prime }+\dfrac{d}{ds}((\phi (s))^{(1/N)})\leq 0
\end{equation}
for almost all $s\in (k_{0},k)$. The left hand side of $(\ref{eq7.7})$ is in $%
L^{1}(k_{0},k)$, hence we integrate $(\ref{eq7.7})$ over $[k_{0},k]$. Moreover,
since $\phi $ is non-increasing, integrating $(\ref{eq7.7})$ over $(k_{0},k)$ we get
\begin{equation}\label{eq7.8}
(\phi (k))^{1/N}\leq \phi (k_{0})^{1/N}+\dfrac{1}{N}C^{\prime \prime
}(k_{0}-k)
\end{equation}
and from $(\ref{eq7.8})$ the contradiction follows.

\section{Example}

This section is devoted to an example for illustrating our abstract result.\\
Let us consider the special case:\\
$$\beta(r)=(r-1)^+ - (r-1)^-,\hspace*{4mm}  F:\mathbb{R} \rightarrow (F_i)_{i=1,...,N} \in \mathbb{R}^N,$$
where $F$ is locally lipshitz continuous function, and 
$$ a_i(x,\xi)= \sum_{i=1}^{N}|\xi_i|^{p_i-1} sgn(\xi_i),\hspace*{4mm} i=1,...,N,$$
the $a_i(x,\xi)$ are Carath\'{e}dory function satisfying the growth condition $(\mathbf{H}_{2})$, and the coercivity $(\mathbf{H}_{1})$. On the other the monotonicity condition is verified. In fact

$$\sum_{i=1}^{N}\Big \lgroup a_i(x,\xi) - a_i(x, \tilde{\xi})\Big \rgroup(\xi_i- \tilde{\xi_i}) = \sum_{i=1}^{N}\Big\lgroup|\xi_i|^{p_i-1}sgn(\xi_i)- |\tilde{\xi_i}|^{p_i-1}sgn(\tilde{\xi}_i)\Big \rgroup(\xi_i - \tilde{\xi_i})\geq 0,$$
for almost all $x\in \Omega$ and for all $\xi, \tilde{\xi} \in \mathbb{R}^N$. This last inequality can not be strict, since for $\xi \neq \tilde{\xi}$ with $\xi_N \neq \tilde{\xi}_N$ and $\xi=\tilde{\xi}, i=1,...,N-1$. The corresponding expression is zero.

Therefore, for all $f\in L^(\Omega),$ the following problem:

$$
   \left\{
\begin{array}{l@{~}l@{~}l@{~}l}
T_k(u)\in W_{0}^{1,{\overrightarrow{p}}}(\Omega )\hspace*{3mm} \text{for} \hspace*{3mm}(k>0); b\in L^1(\Omega)\hspace*{3mm} \text{and}\hspace*{3mm} b(x)\in \beta(u(x)),\\

\lim_{l\rightarrow \infty}\int_{\{l<|u|<l+1\}}a(x,Du).Dudx = 0,\\

\int_\Omega bh(u)\varphi dx + \int_\Omega h(u)\sum_{i=1}^{N}\Big\lvert\dfrac{\partial u}{\partial x_i} \Big\rvert^{p_i-1} sgn \Big(\dfrac{\partial u}{\partial x_i}\Big).\dfrac{\partial \varphi}{\partial x_i}dx\\

+ \int_\Omega h'(u)\sum_{i=1}^{N}\Big\lvert\dfrac{\partial u}{\partial x_i} \Big\rvert^{p_i-1} sgn \Big(\dfrac{\partial u}{\partial x_i}\Big).\dfrac{\partial\varphi}{\partial x_i}dx + \int_\Omega F(u).D(h(u)\varphi) dx\\

= \int_\Omega f.D(h(u)\phi), \hspace*{3mm} \forall \varphi \in  W_{0}^{1,{\overrightarrow{p}}}(\Omega )\cap L^\infty(\Omega) \hspace*{3mm}and \hspace*{3mm} h\in C_c^1(\mathbb{R}),

\end{array}
 \right.$$

at least one renormalized solution.

\end{document}